\documentclass[11pt]{article}
\usepackage{graphicx} 

\title{The Kobayashi-Hitchin correspondence\\ for nef and big classes}
\author{Satoshi Jinnouchi \thanks{Department of Mathematics, Graduate School of Science, The University of Osaka,
1-1, Machikaneyama-cho, Toyonaka, Osaka 560-0043, Japan.
email:{{\tt u122988d[@]ecs.osaka-u.ac.jp}},
email:{{\tt 20160312sti[@]gmail.com}}}}
\date{August 2025}

\usepackage{amssymb,amsmath,amsthm}    
\usepackage[abbrev]{amsrefs} 
\usepackage{mathrsfs}
\usepackage[bbgreekl]{mathbbol}
\usepackage{comment}
\usepackage{color}
\usepackage{tikz}
\usepackage{mathtools}
\usepackage{appendix}
\usepackage[all]{xy} 
\usetikzlibrary{positioning}
\usetikzlibrary {arrows.meta}

\newtheorem{theo}{Theorem}[section]
\newtheorem{lemm}[theo]{Lemma}
\newtheorem{corr}[theo]{Corollary}
\newtheorem{prop}[theo]{Proposition}

\numberwithin{equation}{section}

\theoremstyle{definition}
\newtheorem{defi}[theo]{Definition}
\newtheorem{exam}[theo]{Example}
\newtheorem{rema}[theo]{Remark}

\newtheorem{claim}[theo]{Claim}

\newtheorem{notation}[theo]{Notation}

\allowdisplaybreaks

\newcommand{\rk}{{\rm{rk}}}
\newcommand{\Supp}{{\rm{Supp}}}
\newcommand{\Amp}{{\rm{Amp}}}
\newcommand{\sing}{{\rm{sing}}}
\newcommand{\Sing}{{\rm{Sing}}}
\newcommand{\Exc}{{\rm{Exc}}}
\newcommand{\reg}{{\rm{reg}}}

\newcommand{\codim}{{\rm{codim}}}
\newcommand{\Id}{{\rm{Id}}}
\newcommand{\id}{{\rm{id}}}
\newcommand{\Ker}{{\rm{Ker}}}

\newcommand{\loc}{{\rm{loc}}}

\newcommand{\Hom}{{\rm{Hom}}}
\newcommand{\exc}{{\rm{exc}}}

\newcommand{\Tor}{{\rm{Tor}}}

\newcommand{\End}{{\rm{End}}}

\newcommand{\Tr}{{\rm{Tr}}}

\newcommand{\vol}{{\rm{vol}}}

\newcommand{\wtil}{\widetilde}
\newcommand{\diag}{{\rm diag}}
\newcommand{\im}{{\rm im}}
\newcommand{\triv}{{\rm triv}}

\newcommand{\C}{\mathbb{C}}
\newcommand{\Q}{\mathbb{Q}}

\newcommand{\Sym}{{\rm Sym}}
\newcommand{\delbar}{\bar{\partial}}

\usepackage[colorlinks=true, linkcolor=blue, citecolor=blue]{hyperref}
\usepackage{cite}

\begin{document}
\date{\empty}
\maketitle
\abstract{In this paper, we establish the Kobayashi–Hitchin correspondence for nef and big cohomology classes by introducing the notions of adapted closed positive (1,1)-currents and adapted Hermitian–Yang–Mills metrics. As applications, we investigate the equality cases of both the Bogomolov–Gieseker inequality for semistable reflexive sheaves with respect to big classes admitting a bimeromorphic Zariski decomposition and the Miyaoka–Yau inequality for projective varieties with big anti-canonical divisor.}

\tableofcontents
\section{Introduction}
\setcounter{theo}{0}
\renewcommand{\thetheo}{\Alph{theo}}
The Kobayashi-Hitchin correspondence asserts that a holomorphic vector bundle $E$ over a compact complex manifold $X$ is slope polystable with respect to a K\"{a}hler class $\{\omega\}$ on $X$ if and only if $E$ admits an $\omega$-Hermitian-Yang-Mills metric. If $\omega$ is a smooth K\"{a}hler metric, this correspondence was proved by Donaldson \cite{Don85} and Uhlenbeck-Yau \cite{UY86} for holomorphic vector bundles on compact K\"{a}hler manifolds, by Bando-Siu \cite{BS94} for reflexive sheaves on compact K\"{a}hler manifolds by introducing the notion of admissible Hermitian-Yang-Mills metrics, and by Xuemiao Chen \cite{Chen25} for reflexive sheaves on compact normal K\"{a}hler varieties. More recently, the correspondence has been extended to compact klt K\"{a}hler varieties endowed with orbifold K\"{a}hler metrics motivated by the study of the equality cases of the Bogomolov-Gieseker inequality on klt K\"{a}hler varieties and the Miyaoka-Yau inequality of klt varieties with nef (anti-) canonical divisor \cite{Faulk22}, \cite{CGNPPW23}, \cite{FO25} (refer to the introduction of \cite{IJZ25} for more detail).

In this paper, we generalize these results to a broader class of singular K\"{a}hler metrics, namely those whose cohomology classes are nef and big, not necessarily K\"{a}hler. This generalization allows us to establish the results on the equality cases of the Bogomolov--Gieseker inequality with respect to nef and big classes, as well as the Miyaoka--Yau inequality for projective varieties with big, not necessarily nef, anti-canonical divisor.

\subsection{The Kobayashi-Hitchin correspondence}
Let $X$ be a compact K\"{a}hler manifold, $\alpha$ be a nef and big class on $X$ and $E$ be a holomorphic vector bundle on $X$. In this paper, we firstly introduce the notion of an {\it adapted closed positive $(1,1)$-current} $T$ in $\alpha$. A typical example of an adapted current $T$ is given by a solution to the complex Monge-Amp\`{e}re equation $\langle T^n\rangle=e^f\omega_0^n$ where $f\in C^{\infty}(X)$ is a smooth function and $\omega_0$ is a smooth K\"{a}hler metric on $X$ (refer to Definition \ref{adapted defi} for the precise definition).  We also define the notion of a {\it $T$-adapted Hermitian-Yang-Mills metric} $h$ on $E$, that is an admissible Hermitian-Yang-Mills metric defined by Bando-Siu \cite{BS94}, satisfying two additional conditions: namely $h$ and its first derivatives are less singular than poles of finite order along the non-K\"{a}hler locus of $\alpha$ (refer to Definition \ref{adapted HYM defi} for details).
 
Then the main result of this paper is stated as follows:
\begin{theo}[see Theorem \ref{KH corr nef big with esti}, Theorem \ref{HYM to stable and unique}]\label{main thm2 intro}
Let $X$ be a compact K\"{a}hler manifold and $\alpha$ be a nef and big class on $X$. Let $T$ be an adapted closed positive $(1,1)$-current in $\alpha$. Let $E$ be a holomorphic vector bundle over $X$. Then, the following conditions are equivalent:
\begin{enumerate}
\item $E$ is $\alpha^{n-1}$-slope polystable (refer to Definition \ref{stability defi}).
\item $E$ admits a $T$-adapted Hermitian-Yang-Mills metric.
\end{enumerate}
Furthermore, if $E$ admits a $T$-adapted Hermitian-Yang-Mills metric, then it is unique up to scaling.
\end{theo}
By \cite{CCHSTT25}, it has been proved that the pull-back of a singular K\"{a}hler-Einstein metric $\omega$ on a compact klt variety $Y$ along a resolution of singularities $\pi:X\to Y$ is an adapted current. Hence above Theorem \ref{main thm2 intro} applies to this $X$, a nef and big class $\alpha=\{\pi^*\omega\}$, an adapted current $T=\pi^*\omega$ and any holomorphic vector bundle $E$ on $X$.

Since the direct sum, the tensor product and the wedge product of $T$-adapted HYM metrics are again $T$-adapted HYM, we obtain the following result as a consequence of Theorem \ref{main thm2 intro}.
\begin{corr}
Let $X$ be a compact K\"{a}hler manifold and $\alpha$ be a nef and big class. Then, for any $\alpha^{n-1}$-slope polystable holomorphic vector bundles ${E}_1$ and ${E}_2$, their  tensor product ${E}_1\otimes{E}_2$ and wedge product ${E}_1\wedge{E}_2$ are $\alpha^{n-1}$-slope polystable. If furthermore $E_1$ and $E_2$ have the same $\alpha^{n-1}$-slope, then the direct sum ${E}_1\oplus {E}_2$ is also $\alpha^{n-1}$-slope polystable. 
\end{corr}

\subsection{The Bogomolov-Gieseker inequality and the Miyaoka-Yau inequality}
We say that a big class $\alpha\in H^{1,1}_{BC}(X)$ on a compact normal variety of Fujiki class admits a bimeromorphic Zariski decomposition if there exists a resolution $\pi:Y\to X$ so that $Y$ is smooth K\"{a}hler and $\langle\pi^*\alpha\rangle$ is nef and big (refer to Definition \ref{bimero zar decomp defi}). Following \cite{Chen25}, we define the discriminant of a reflexive sheaf $\mathcal{E}$ on a compact normal Moishezon variety $X$ with respect to a big class $\alpha\in H^{1,1}_{BC}(X)$ admitting a bimeromorphic Zariski decomposition as follows (see Definition \ref{BG discri defi}):
\begin{equation}\label{BG discri intro}
\Delta(\mathcal{E})\cdot\langle\alpha^{n-2}\rangle
:=\inf_{\pi:Y\to X}\left(2rc_2(\pi^*\mathcal{E}/\Tor)-(r-1)c_1(\pi^*\mathcal{E}/\Tor)^2\right)\cdot\langle(\pi^*\alpha)^{n-2}\rangle,
\end{equation}
where $\pi:Y\to X$ is a resolution such that $Y$ is a smooth projective variety, $\pi^*\mathcal{E}/\Tor$ is locally free and $\langle\pi^*\alpha\rangle$ is nef and big.
Then, we obtain the following result as an application of Theorem \ref{main thm2 intro}:
\begin{theo}[= Theorem \ref{prop-semistable}, cf.\cite{IJZ25}]\label{big BG eq intro} 
Let $X$ be a compact normal Moishezon variety, $\mathcal{E}$ be a reflexive sheaf of rank $r$ and $\alpha \in H^{1,1}_{BC}(X)$ be a big class admitting a bimeromorphic Zariski decomposition. 
\begin{enumerate}
\item If $\mathcal{E}$ is $\langle\alpha^{n-1}\rangle$-slope polystable, then the Bogomolov-Gieseker inequality 
\begin{equation}\label{BG ineq big intro}
\Delta(\mathcal{E})\langle\alpha^{n-2}\rangle\ge0
\end{equation}
holds. If $\Delta(\mathcal{E})\langle\alpha^{n-2}\rangle=0$ holds, then $\mathcal{E}$ is locally free on $\Amp(\alpha)$ and projectively flat on $\Amp(\alpha)$.
\item If $\mathcal{E}$ is $\langle\alpha^{n-1}\rangle$-slope semistable and it satisfies $\Delta(\mathcal{E})\langle\alpha^{n-2}\rangle=0$, then the Jordan-H{$\ddot{o}$}lder filtration of $\mathcal{E}$,
$$
0\subset \mathcal{E}_{k}\subset\cdots\subset\mathcal{E}_1\subset \mathcal{E},
$$
satisfies that $\mathcal{E}_{i}/\mathcal{E}_{i+1}|_{\Amp(\alpha)}$ is projectively flat.
\end{enumerate}
\end{theo}
If $\alpha$ is nef and big and $X$ is smooth, the same result holds for, possibly non Moishezon, K\"{a}hler manifolds (see Theorem \ref{BG eq}).
 The Bogomolov-Gieseker inequality (\ref{BG ineq big intro}) was proved by \cite{IJZ25} for general big cohomology classes. The projectively flatness in the equality of the Bogomolov-Gieseker inequality as above Theorem \ref{big BG eq intro} is even new if $X$ is a projective manifold and $\alpha$ is the 1st Chern class of a nef and big line bundle. 
\begin{rema}[see Example \ref{BG exam}]
We also construct a simple example of a vector bundle which does not satisfy the Bogomolov-Gieseker equality with respect to K\"{a}hler classes, but does satisfy it with respect to a nef and big class.
\end{rema} 
As an application of Theorem \ref{big BG eq intro}, we obtain the following.
 \begin{theo}[= Theorem \ref{MY eq}]\label{MY eq intro}
Let $X$ be an $n$-dimensional projective klt variety smooth in codimension $2$. Assume that $X$ is K-stable and that $-K_X$ is big. If the equality
\begin{equation}
\label{eq-MY-equality intro}
\bigl(2(n+1)c_2(X) - nc_1(X)^2 \bigr) \cdot \langle c_1(-K_{X})^{n-2} \rangle = 0
\end{equation}
holds, then the anticanonical model $Z$ admits a quasi-étale cover from $\mathbb{C}\mathbb{P}^n$. $($In this case, by \cite{Xu23}, the anticanonical model $Z$ exists.$)$
\end{theo}
There are several studies about the Miyaoka-Yau inequality and the structure theorem in the case that the equality holds (refer to the introduction of \cite{IJZ25} and references therein for detailed history and related developments concerning the Miyaoka-Yau inequality). We briefly summarize the references of known results. If $X$ is smooth and $K_X$ is ample, refer to \cite{Miy77} when $n=\dim X=2$ and \cite{Yau77} when $X$ is of arbitrary dimension. If $K_X$ is nef, refer to \cite{GKPT19} when $X$ is klt and \cite{GT22} when $(X,D)$ is dlt. If $X$ is projective klt and $-K_X$ is nef, refer to \cite{GKP22}. If $X$ is klt and $K_X$ is big, or $-K_X$ is big and $X$ is $K$-semistable, refer to \cite{IJZ25}. In particular, the authors only proved the inequality in \cite{IJZ25}.
If $X$ is klt and $K_X$ is big, refer to \cite{ZZZ26}.
 Theorem \ref{MY eq intro} in this paper proves the structure theorem of projective varieties with $-K_X$ big. 

\section*{Acknowledgment}
The author would like to thank his supervisor Ryushi Goto for his helpful comments and nice discussions. The author also thanks Prof.  Yoshinori Hashimoto and Prof. Masataka Iwai for essential comments and discussions. The author would like to thank Prof. Yuji Odaka for introducing him to the important paper \cite{CCHSTT25}. The author is grateful to Prof. Yuta Watanabe and Rei Murakami for fruitful discussions and important comments.

\section{Preliminaries and Definitions}
\renewcommand{\thetheo}{\thesection.\arabic{theo}}
\subsection{Nonpluripolar product}
\subsubsection{Positivities of cohomology classes}
Let $X$ be a compact K\"{a}hler manifold with a smooth K\"{a}hler metric $\omega$. 
In this paper, we will focus on the positivities of cohomology classes in the following sense.
\begin{defi}
\begin{enumerate}
\item  A cohomology class $\alpha\in H^{1,1}(X,\mathbb{R})$ is pseudo-effective if $\alpha$ is represented by a closed positive $(1,1)$-current.
\item A cohomology class $\alpha\in H^{1,1}(X,\mathbb{R})$ is big if $\alpha$ is represented by a K\"{a}hler current $T$, that is, there exists $\varepsilon>0$ such that $T\ge \varepsilon\omega$ in the sense of currents.
\item A cohomology class $\alpha$ is nef if for any $\varepsilon>0$ there exists a smooth $(1,1)$-form $\alpha_{\varepsilon}$ in $\alpha$ such that $\alpha_{\varepsilon}\ge -\varepsilon\omega$.
\end{enumerate}
\end{defi}
\noindent It is well-known that a nef cohomology class $\alpha$ is big if and only if its volume $\alpha^n$ is positive \cite{DP04}. 
\begin{defi}[\cite{Bou04}]
Let $\alpha$ be a pseudo-effective class on $X$. Then, the non-K\"{a}hler locus of $\alpha$ is defined as 
$$
E_{nK}(\alpha):=\{x\in X\mid \hbox{There is a K\"{a}hler current $T\in\alpha$ which is smooth K\"{a}hler around $x$}\}
$$
The ample locus of $\alpha$ is the complement of the non-K\"{a}hler locus of $\alpha$, that is, $\Amp(\alpha):=X\setminus E_{nK}(\alpha)$.
\end{defi}
In singular settings, we define the non-K\"{a}hler locus in the following way:
\begin{defi}\label{nK locus defi sing}
Let $X$ be a compact normal variety which is bimeromorphic to a compact K\"{a}hler manifold and $\alpha\in H^{1,1}_{BC}(X)$ be a pseudo-effective class on $X$. Let $\pi:Y\to X$ be a resolution of singularities of $X$ so that $Y$ is smooth K\"{a}hler. Then, we define the non-K\"{a}hler locus of $\alpha$ by 
$$
E_{nK}(\alpha):=\pi(E_{nK}(\pi^*\alpha)).
$$
The ample locus of $\alpha$ is defined by $\Amp(\alpha):=X\setminus E_{nK}(\alpha)$.
\end{defi}
\begin{rema}
In the above Definition \ref{nK locus defi sing}, we can see that $X_{\sing}\subset E_{nK}(\alpha)$.
\end{rema}
\noindent Boucksom \cite[Theorem 3.17]{Bou04} proved that the non-K\"{a}hler locus of big class is an analytic subset of $X$.
By Demailly\rq{s} regularization, the ample locus of any big class is non-empty. Furthermore, a big class $\alpha$ is K\"{a}hler if and only if $\Amp(\alpha)=X$.
The following result by Collins-Tosatti \cite{CT15} plays an important role:
\begin{theo}[\cite{CT15}]\label{CT}
Let $X$ be a compact K\"{a}hler manifold and $\alpha$ be a nef and big class. Let $D$ be a divisor contained in $E_{nK}(\alpha)$ the non-K\"{a}hler locus of $\alpha$. Then the following holds:
$$
\alpha^{n-1}[D]=0.
$$
\end{theo}

\subsubsection{Nonpluripolar product}
In this subsection, we review the general definition of nonpluripolar product of closed positive $(1,1)$-currents following \cite{BEGZ10} (refer to \cite{IJZ25} for definition in singular settings).
Let $T$ be a closed positive $(1,1)$-current in a pseudo-effective class $\alpha$. Let $\varphi$ be a plurisubharmonic function such that $T=dd^c\varphi$ locally. Bedford-Taylor \cite{BT82} proved that if $\varphi$ is locally bounded, then we can define the wedge product $T^p$ inductively:
$$
\langle (dd^c\varphi)^p, \eta\rangle:=\langle \varphi(dd^c\varphi)^{p-1}, dd^c\eta\rangle
$$
where $\eta$ is a compactly supported smooth form and $\langle\cdot,\cdot\rangle$ here is the natural pairing between currents and smooth forms. Then the nonpluripolar product of positive currents in general are defined as the limit of the Bedford-Taylor\rq{s} wedge product on the locally bounded locus. More precisely, the local nonpluripolar product of closed positive $(1,1)$-current $T=dd^c\varphi$ is defined as 
$$
\langle (dd^c\varphi)^p\rangle:=\lim_{k\to\infty}1_{\{\varphi>-k\}}(dd^c\max{\{\varphi,-k\}})^p,
$$
where the limit is the weak limit. We say that the nonpluripolar product of $T$ is well defined if it has locally finite mass. On compact complex, not necessarily K\"{a}hler, manifolds, it is not known whether the nonpluripolar product of any closed positive $(1,1)$-currents is well defined or not. However, the following was proved by \cite{BEGZ10}.
\begin{prop}[\cite{BEGZ10}, Proposition 1.6, Theorem 1.8]
Let $X$ be a compact K\"{a}hler manifold and $T$ be a closed positive $(1,1)$-current in a pseudo-effective class $\alpha$. Then the nonpluripolar product of $T$, denoted by $\langle T^p\rangle$, is always globally well-defined and it is a closed positive $(p,p)$-current on $X$. 
\end{prop} 
Next we introduce a simple description of nonpluripolar products in a special setting which is used in this paper.
Let $T_i$ be a closed positive $(1,1)$-current on $X$ ($i=1,2$) contained in a pseudo-effective cohomology class $\alpha$. We fix a smooth representative $\theta$ of $\alpha$. Then, there exists a $\theta$-plurisubharmonic function $\varphi_i$ on $X$, called as a potential of $T_i$, such that $T_i=\theta+dd^c\varphi_i$. Then we say that $T_1$ is less singular than $T_2$ if there is a bounded function $f$ on $X$ such that $\varphi_2\le \varphi_1+f$ holds. We say a closed positive $(1,1)$-current $T$ in a pseudo-effective class $\alpha$ has minimal singularities if $T$ is less singular than any other closed positive $(1,1)$-current in $\alpha$. If $\alpha$ is a big class, then it contains a K\"{a}hler current on $X$ which is smooth K\"{a}hler on the ample locus $\Amp(\alpha)$ (\cite[Theorem 3.17]{Bou04}). Hence, any potential of a closed positive $(1,1)$-current $T$ with minimal singularities in a big class $\alpha$ is locally bounded on $\Amp(\alpha)$.   Since the nonpluripolar products put no mass on any pluripolar subset, we can see that if $T$ has minimal singularities in a big class $\alpha$, then
$$
\langle T^p\rangle =1_{\Amp(\alpha)}(T|_{\Amp(\alpha)})^p
$$
where the RHS above is the zero extension from $\Amp(\alpha)$ to $X$ of a closed positive $(p,p)$-current $(T|_{\Amp(\alpha)})^p$ on $\Amp(\alpha)$. 
\begin{defi}[\cite{BEGZ10}, Definition 1.17]
Let $\alpha$ be a pseudo-effective class on a compact K\"{a}hler manifold $X$. Then, the positive product of $\alpha$ is defined as the cohomology class of the nonpluripolar product of closed positive $(1,1)$-current $T_{\min}$ with minimal singularities in $\alpha$, that is,
$$
\langle \alpha^p\rangle:=\{\langle T_{\min}^p\rangle\}\in H^{p,p}(X,\mathbb{R}).
$$
\end{defi}
\noindent By \cite[Theorem 1.16]{BEGZ10}, a pseudo-effective class $\alpha$ is big if $\langle\alpha^n\rangle>0$. Furthermore, if a big class $\alpha$ is nef, then $\langle\alpha^p\rangle=\alpha^p$ for any $p>0$ \cite{BEGZ10}.
In this paper, we will use the following general result.
\begin{theo}[\cite{BEGZ10}, Theorem 3.1, Theorem 4.1, Theorem 5.1, \cite{GPTW24}]\label{MA BEGZ}
Let $X$ be a compact K\"{a}hler manifold and $\alpha$ be a big class on $X$. Then the following holds:
\begin{enumerate}
\item If $\mu$ is a nonpluripolar measure on $X$ with $\mu(X)=\langle\alpha^n\rangle$, then there exists a closed positive $(1,1)$-current $T$ in $\alpha$ such that $T$ satisfies the following complex Monge-Amp\`{e}re equation: 
\begin{equation}\label{MA eq}
\langle T^n\rangle=\mu.
\end{equation}
\item If $\mu=e^fdV$ where $dV$ is a smooth volume form on $X$ and $e^f\in L^p(X)$ with $p>1$, then the solution $T$ of (\ref{MA eq}) has minimal singularities in $\alpha$.
\item If $\alpha$ is nef and big and $f\in C^{\infty}(X)$, then the solution $T$ of (\ref{MA eq}) is smooth on $\Amp(\alpha)$.
\item Assume that $\alpha$ is a nef and big class and $\theta$ be a smooth $(1,1)$-form in $\alpha$. Let $T=\theta+dd^c\varphi$ be a closed positive $(1,1)$-current in $\alpha$ which satisfies $\sup_X\varphi=0$ and
$$
\langle T^n\rangle=e^f\omega_0^n
$$
where $e^f\in L^p(X)$ with $p>1$. Let $\omega_i:=(\theta+(1/i)\omega_0)+dd^c\varphi_i$ be a smooth K\"{a}hler metric in a K\"{a}hler class $\alpha+(1/i)\omega_0$ given by the solution to
$$
\omega_i^n=e^{f_i}\omega_0^n.
$$
Assume that $f_i\in C^{\infty}(X)$ and $\int_X|f_i|^pe^{f_i}\omega_0^n\le C$ for any $i$. Then, there exists a constant $C>0$ such that the following holds for any $i$:
$$
0\le -\varphi_i+V_i\le C,
$$
where $V_i=\sup\{\psi\mid \hbox{ nonpositive $(\theta+(1/i)\omega_0)$-plurisubharmonic function on $X$}\}$. In particular, if $D=E_{nK}(\alpha)$ is a snc divisor, then the following holds:
$$
\varphi_i\ge \log|s_D|^2-C.
$$
\end{enumerate}
\end{theo}

\subsection{Definition of \texorpdfstring{$\alpha^{n-1}$}{alpha{n-1}}-slope polystability}
In this section, we review the definition of slope stability for nef and big classes and its important properties. The precise definition of slope polystability for nef and big classes was firstly given by \cite[Definition 3.2]{Jin25-2}
\begin{defi}\label{stability defi}
Let $X$ be a compact K\"{a}hler manifold and $\alpha$ be a nef and big class on $X$. Let $\mathcal{E}$ be a torsion free coherent sheaf on $X$. 
\begin{enumerate}
\item We define the $\alpha^{n-1}$-slope of $\mathcal{E}$, denoted by $\mu_{\alpha}(\mathcal{E})$, by 
$$
\mu_{\alpha}(\mathcal{E}):=\frac{1}{\rk \mathcal{E}}\int_Xc_1(\mathcal{E})\wedge\alpha^{n-1}.
$$
\item A torsion free sheaf $\mathcal{E}$ is $\alpha^{n-1}$-slope stable (resp.$\alpha^{n-1}$-slope semistable) if for any nontrivial torsion free subsheaf $\mathcal{F}\subset \mathcal{E}$, the inequality 
$$
\mu_{\alpha}(\mathcal{F})< \mu_{\alpha}(\mathcal{E}) \hspace{2mm}(\hbox{ resp. } \mu_{\alpha}(\mathcal{F})\le \mu_{\alpha}(\mathcal{E}))
$$
holds.
\item A torsion free sheaf $\mathcal{E}$ is $\alpha^{n-1}$-slope polystable if there exists $\alpha^{n-1}$-slope stable torsion free sheaves $\mathcal{E}_1,\ldots,\mathcal{E}_k$ on $X$ with $\mu_{\alpha}(\mathcal{E}_1)=\cdots=\mu_{\alpha}(\mathcal{E}_k)$ such that there exists an isomorphism 
$$
\mathcal{E}|_{\Amp(\alpha)}\simeq (\mathcal{E}_1\oplus\cdots\oplus\mathcal{E}_k)|_{\Amp(\alpha)}
$$
as coherent sheaves on $\Amp(\alpha)$.
\end{enumerate}
\end{defi} 
In singular settings, we define $\alpha^{n-1}$-slope stability via resolution of singularities as follows.
\begin{defi}\label{stable defi sing}
Let $X$ be a compact normal analytic variety which is bimeromorphic to a compact K\"{a}hler manifold and $\alpha\in H^{1,1}_{BC}(X)$ be a nef and big class on $X$. Let $\pi:Y\to X$ be a resolution of singularities of $X$ so that $Y$ is a compact K\"{a}hler manifold. Then a torsion free sheaf $\mathcal{E}$ on $X$ is $\alpha^{n-1}$-slope stable (resp. semistable, polystable) if $\pi^*\mathcal{E}/\Tor$ is $(\pi^*\alpha)^{n-1}$-slope stable (resp. semistable, polystable).
\end{defi}
\noindent By the following lemma, above Definition \ref{stable defi sing} is independent of the choice of a resolution $\pi$.  
\begin{lemm}[\cite{IJZ25}, Proposition 3.28]\label{inv stability}
Let $X$ be a compact K\"{a}hler manifold and $\alpha$ be a nef and big class on $X$. Let $\mathcal{E}$ be a torsion free sheaf on $X$. Let $\pi:Y\to X$ be a composition of blow-ups along complex submanifolds contained in $E_{nK}(\alpha)$ or of codimension at least 2. Then $\mathcal{E}$ is $\alpha^{n-1}$-slope polystable if and only if $\pi^*\mathcal{E}/\Tor$ is $(\pi^*\alpha)^{n-1}$-slope polystable. In particular, Definition \ref{stable defi sing} is independent of the choice of a resolution $\pi$.
\end{lemm}
The proof of our main Theorem \ref{main thm2 intro} relies on the openness of $\alpha^{n-1}$-slope stability in the following sense.
\begin{lemm}[\cite{Jin25-2}, Lemma 4.5, see also \cite{Ou25} Claim 3.1]\label{open stable}
Let $X$ be a compact \\ K\"{a}hler manifold with a K\"{a}hler class $\omega_0$ and $\alpha$ be a nef and big class. If a holomorphic vector bundle $E$ is $\alpha^{n-1}$-slope stable, then $E$ is $(\alpha+\varepsilon\omega_0)^{n-1}$-slope stable for sufficiently small $\varepsilon>0$.
\end{lemm}
\begin{proof}
We begin to start with the following:
\begin{claim}\label{max slope}
$(1)$ Let $\eta\in H^{2n-2}(X,\mathbb{R})$ be a cohomology class which is represented by a positive $(2n-2)$-current.
Then there is a constant $C>0$ such that for any subsheaf $\mathcal{F}\subsetneq E$, the following inequality holds:
$$
\int_Xc_1(\mathcal{F})\wedge\eta\le C.
$$
\noindent $(2)$ There is a nontrivial torsion free subsheaf $\mathcal{F}_0\subsetneq E$ such that
$$
\mu_{\alpha}(\mathcal{F}_0) = \sup\{\mu_{\alpha}(\mathcal{F})\mid 0\neq \mathcal{F}\subsetneq E \hbox{ is torsion free}\}.
$$
\end{claim}
\begin{proof}
$(1)$ The inclusion $\mathcal{F}\subset E$ induces a sheaf morphism $\iota:\det\mathcal{F}\to \bigwedge^{\rk\mathcal{F}}E$ which is injective on the locally free locus of $\mathcal{F}$ which is of codimension at least $3$. Since $\det\mathcal{F}$ is torsion free, $\iota$ is injective. Furthermore $\det\mathcal{F}$ is a subbundle of $\bigwedge^{\rk \mathcal{F}}E$ since $\det\mathcal{F}$ is a subbundle away from an analytic subset of codimension 3 and its saturation gives a subbundle.
Let $h$ be a hermitian metric on $\bigwedge^{\rk\mathcal{F}} E$. We denote by $h_{\det\mathcal{F}}$ the restriction of $h$ to $\det\mathcal{F}$. Then the curvature of $h_{\det\mathcal{F}}$ satisfies $F_{h_{\det\mathcal{F}}}=p\circ F_h \circ p + \delbar p\wedge \partial p$ where $p:\bigwedge^{\rk \mathcal{F}}E\to \bigwedge^{\rk\mathcal{F}}E$ is the $h$-orthogonal projection to $\det\mathcal{F}$ (e.g. \cite{Kob87}).  Thus we can calculate as follows:
\begin{align*}
\int_Xc_1(\mathcal{F})\wedge\eta
&=\int_X\sqrt{-1}\Tr(p\circ F_h \circ p + \delbar p\wedge \partial p)\wedge\eta\\
&\le \sqrt{-1}\int_X\Tr(p\circ F_h\circ p)\wedge \eta\\
&\le\|F_h\|_{L^{\infty}}\rk(\mathcal{F})\int_X\omega_0\wedge\eta=C.
\end{align*}
$(2)$ Since we only consider the supremum, we can assume $-C< \mu_{\alpha}(\mathcal{F})$ for $\mathcal{F}\subsetneq E$. We can find elements $\eta_1,\ldots, \eta_k\in H^{2n-2}(X,\mathbb{Q})$ which forms a basis of $H^{2n-2}(X,\mathbb{Q})$ such that
\begin{itemize}
\item each $\eta_i$ is represented by a positive $(2n-2)$-current,
\item $\eta_1,\ldots, \eta_{k-1}$ lies near $\alpha^{n-1}$ and
\item there are nonnegative constants $a_1,\ldots, a_k\ge 0$ such that $\alpha^{n-1}=\Sigma_{i=1}^{k-1}a_i\eta_i-a_k\eta_k.$
\end{itemize}
These $\eta_i$ exist since $H^{2n-2}(X,\mathbb{Q})$ is dense in $H^{2n-2}(X,\mathbb{R})$ and the set $\{\alpha^{n-1}-\Sigma_{i=1}^{k-1}a_i\eta_i\mid a_i\in (1-\varepsilon,1+\varepsilon)\}$ is an open set.
By $(1)$, we have a constant $C>0$ so that $-C<\mu_{\alpha}(\mathcal{F})< C$ holds for  subsheaf $\mathcal{F}$, since $\alpha^{n-1}$ is represented by a positive current $\langle T_{\min}^{n-1}\rangle$ where $T_{\min}$ is a closed positive $(1,1)$-current with minimal singularities in $\alpha$.  Thus we also have $-C< \mu_{\eta_i}(\mathcal{F})<C$ for any $i=1,\ldots k-1$ and $\mathcal{F}\subsetneq E$ since for each $\mathcal{F}$, the slope $\mu_{\eta}(\mathcal{F})$ is a continuous function in $\eta\in H^{2n-2}(X,\mathbb{R})$. It implies the boundedness $-C<\mu_{\eta_k}(\mathcal{F})<C$. In fact, the upperboundedness follows from $(1)$. If there is no lower bound, there is a sequence of subsheaves $(\mathcal{F}_j)_j$ such that $\mu_{\eta_k}(\mathcal{F}_j)$ monotonically decreasing to $-\infty$. Then, the inequality
$-C<\mu_{\alpha}(\mathcal{F})=\Sigma_{i=1}^{k-1}a_i\mu_{\eta_i}(\mathcal{F})-a_k\mu_{\eta_k}(\mathcal{F})<C$ implies $\Sigma_{i=1}^{k-1}a_i\mu_{\eta_i}(\mathcal{F})$ monotonically decreases to $-\infty$. Since each $a_i$ is non-negative, it contradicts the existence of a lower bound of $\mu_{\eta_i}(\mathcal{F})$ for $i=1,\ldots, k-1$. Since $\eta_i$ are $\mathbb{Q}$-cohomology classes, there is a positive integer $m$ such that $m\eta_i$ are $\mathbb{Z}$-cohomology classes. Of course $\mu_{m\eta_i}(\mathcal{F})$ is bounded. We also remark that $\int_Xc_1(\mathcal{F})\wedge m\eta_i$ is an integer. Then, we can see that
$$
\{\mu_{m\eta_i}(\mathcal{F})\mid 0\neq \mathcal{F}\subsetneq E, \mu_{\eta_i}(\mathcal{F})> -C\}\subset [-C,C]\cap \mathbb{Z}
$$
and the LHS is a finite set. Thus we obtain that
$$
\{\mu_{\alpha}(\mathcal{F})\mid 0\neq \mathcal{F}\subsetneq E, \mu_{\alpha}(\mathcal{F})> -C\}
$$
is a finite set and thus there is a nontrivial subsheaf $\mathcal{F}_0\subsetneq E$ which attains the supremum of the $\alpha^{n-1}$-slope.
\end{proof}
\noindent We remember that $E$ is $\alpha^{n-1}$-slope stable. Let us define K\"{a}hler classes $\alpha_{\varepsilon}:=\alpha+\varepsilon\omega_0$ for small $\varepsilon>0$.
The we snow that 
$E$ is $\alpha_{\varepsilon}^{n-1}$-slope stable for small $\varepsilon>0$.
Let $0\neq\mathcal{F}_0\subsetneq E$ be a torsion free subsheaf with maximal $\alpha^{n-1}$-slope in Lemma \ref{max slope} $(2)$. Then there is a large $\varepsilon_0>0$ which is independent of subsheaves of $E$ such that for any $0<\varepsilon<\varepsilon_0$ and $0\neq \mathcal{F}\subsetneq E$, the following inequality holds:
\begin{align*}
&\mu_{\alpha_{\varepsilon}}(E)-\mu_{\alpha_{\varepsilon}}(\mathcal{F})\\
&=\mu_{\alpha}(E)-\mu_{\alpha}(\mathcal{F})+\Sigma_{j=1}^{n-1}C_j\varepsilon^j\left(\frac{\int_Xc_1(E)\wedge\alpha^{n-1-j}\wedge\omega_0^j}{\rk E}-\frac{\int_Xc_1(\mathcal{F})\wedge\alpha^{n-1-j}\wedge\omega_0^j}{\rk \mathcal{F}} \right)\\
&\ge \mu_{\alpha}(E)-\mu_{\alpha}(\mathcal{F}_0)-\varepsilon C\\
&>0.
\end{align*}
Here the inequality in the third line follows from Lemma \ref{max slope} since $\alpha^{n-1-j}\wedge\omega_0^j$ is represented by a positive current $\langle T_{\min}^{n-1-j}\rangle\wedge\omega_0^j$, where $T_{\min}$ is a closed positive $(1,1)$-current with minimal singularities in $\alpha$. And the last inequality comes from the $\alpha^{n-1}$-slope stability of $E$.
\end{proof}
\begin{exam}[see also Example \ref{BG exam}]\label{nonopenness semistable}
In constrast to the openness of $\alpha^{n-1}$-slope stability Lemma \ref{open stable}, the $\alpha^{n-1}$-slope polystability is not open in the sense of Lemma \ref{open stable}.
Let $\pi:X\to \mathbb{C}P^3$ be a blow up at a point with the exceptional divisor $D$. Let $\alpha:=\pi^*c_1(\mathcal{O}(1))$ be a nef and big class on $X$ and $\omega_{\varepsilon}:=\alpha-\varepsilon[D]$ be K\"{a}hler classes on $X$. We define a holomorphic vector bundle $E$ on $X$ by $E:=\mathcal{O}_X\oplus\mathcal{O}(D)$. Then $E$ is $\alpha^2$-slope polystable, but is not $\omega_{\varepsilon}^2$-slope semistable. In fact the exceptional divisor $\mathcal{O}(D)$ gives a destabilizing subsheaf.
\end{exam}
\begin{rema}
In \cite{Jin25}, the author defined the notion of slope stability if $\alpha$ is big, not necessarily nef, cohomology class by replacing $\alpha^{n-1}$ by the positive product $\langle\alpha^{n-1}\rangle$ \cite[Definition 3.22]{IJZ25}. If $\alpha$ satisfies the vanishing property in the sense of \cite[Definition 3.23]{IJZ25}, in particular if $X$ is Moishezon, then Lemma \ref{inv stability} also holds by \cite[Proposition 3.28]{IJZ25} (refer to the proof of Proposition \ref{prop-semistable}). 
\end{rema}
\subsection{Adapted closed positive \texorpdfstring{$(1,1)$}{(1,1)}-currents}
In this subsection, we introduce the notion of adapted closed positive $(1,1)$-currents which plays an essential role in this paper. The following definition is inspired by the tame approximations of closed positive $(1,1)$-currents with bounded potentials by \cite{CCHSTT25}.
\begin{defi}\label{adapted defi}
Let $X$ be a compact K\"{a}hler manifold and $\alpha$ be a nef and big class on $X$. Then, a closed positive $(1,1)$-current $T$ in $\alpha$ is said to be adapted if it satisfies the following conditions:
\begin{enumerate}
\item $T$ is smooth K\"{a}hler on $\Amp(\alpha)$,
\end{enumerate}
Let $\pi:Y\to X$ be a composition of blow-ups so that $D=E_{nK}(\pi^*\alpha)$ is a snc divisor. We denote by $s_D$ a defining section of $D$ and $h_D$ a smooth hermitian metric on $\mathcal{O}(D)$ such that $|s_D|_{h_D}\le 1$. Then
\begin{enumerate}
\setcounter{enumi}{1}
\item $\pi^*T\ge |s_D|^{2m}_{h_D}\omega_Y$ for some $m\in\mathbb{Z}_{\ge 0}$ where $\omega_Y$ is a smooth K\"{a}hler metric on $Y$.
\item There exists a sequence of smooth K\"{a}hler metrics $\omega_i$ lying in a K\"{a}hler class $\pi^*\alpha+(1/i)\omega_Y$ on $Y$ such that $\omega_i$ locally smoothly converges to $\pi^*T$ on $\Amp(\pi^*\alpha)=Y\setminus D$ and 
$$
\int_Y\left(\frac{\omega_i^n}{\omega_Y^n}\right)^p\omega_Y^n\le C
$$
holds for some constant $C>0$ and $p>1$ both independent of $i$. We say the above $\omega_i$ by an adapted approximation of $T$.
\end{enumerate}
\end{defi}
The following proposition shows that any nef and big class on a compact K\"{a}hler manifold contains an adapted current.
\begin{prop}[\cite{Jin25-2}, Lemma 4.4, see also \cite{BEGZ10}, Theorem 5.1]\label{adapted exists}
Let $X$ be a compact K\"{a}hler manifold with a smooth K\"{a}hler metric $\omega_0$ on $X$ and $\alpha$ be a nef and big class on $X$. Let $T$ be a closed positive $(1,1)$-current in $\alpha$ defined by $\langle T^n\rangle =e^f\omega_0^n$ with $f\in C^{\infty}(X)$. Then $T$ is an adapted current. In particular, any nef and big class on a compact K\"{a}hler manifold contains an adapted current.
\end{prop}
\begin{proof}
Since the result follows from the proof of \cite{BEGZ10} Theorem 5.1, we only sketch the proof.
Let $T$ be a closed positive $(1,1)$-current in $\alpha$ with minimal singularities which satisfies the complex  Monge-Amp\`{e}re equation
$$
\langle T^n\rangle=e^f\omega_0^n
$$
where $f\in C^{\infty}(X)$. By \cite{BEGZ10} (refer to Theorem \ref{MA BEGZ}), such $T$ exists and it is smooth K\"{a}hler on  $X\setminus D$.  We denote by $T=\theta+dd^c\varphi$. Let us define a smooth form by $\theta_i:=\theta+(1/i)\omega_0$ in a K\"{a}hler class $\alpha+(1/i)\omega_0$ and consider a smooth K\"{a}hler metric $\omega_i=\theta_i+dd^c\varphi_i$ defined by
$$
\omega_i^n=c_i(e^f+\frac{1}{i})\omega_0^n.
$$ 
Such $\omega_i$ exists by \cite{Yau78}.
We also consider $\varphi_{D'}$ defined by $[D']=\theta-\omega_0+dd^c\varphi_{D'}$. Then $\varphi_{D'}\ge \log|s_D|^{2k}-C$ for some $k>0$.
Here we remark that $f$ is in particular upper bounded.
Then, by the proof of \cite{BEGZ10} Theorem 5.1, we have
$$
\sqrt{-1}\Lambda_{\omega_0}dd^c\varphi_i=\Delta_{\omega_0}\varphi_i\le C_1\exp(C(C_2-\varphi_{D'}))\le C|s_D|^{-2k}
$$
for any $i>0$. Since $\theta_i\le \theta_1$ is smooth, we obtain $\sqrt{-1}\Lambda_{\omega_0}\omega_i\le C|s_D|^{-2k}$. Let us consider $\omega_0=\Sigma_{k=1}^n\sqrt{-1}dz_k\wedge d\overline{z_k}$ and $\omega_i=\Sigma_{k=1}^n\sqrt{-1}g_{i,k}dz_k\wedge d\overline{z_k}$ at a point $p$. Then we have $g_{i,k}\le C|s_D|^{-2k}$. Since $\omega_i$ satisfies the Monge-Amp\`{e}re equation, we have $g_{i,k}\ge C|s_D|^{2m}$ for some $m>0$. Hence we have $\omega_i\ge C|s_D|^{2m}\omega_0$ for all $i$. 

In \cite{BEGZ10}, they showed that $\varphi_i\to \varphi$ in $L^1$-topology. Let us fix an Euclidean ball $B$ in a coordinate neighborhood in ${X}\setminus D$. Then, \cite{BEGZ10} showed that there exists a constant $C_B>0$ such that the uniform estimate $\|\Delta_{\omega_0}\varphi_i\|_{L^{\infty}(B)}\le C_B$ holds for any $i$. It is also shown by \cite{GPTW24} that $\|\varphi_i\|_{L^{\infty}(B)}\le C_B$. Then, by the standard elliptic estimate (e.g. \cite{GT} Theorem 9.11), we obtain the uniform $L^{p}_2$-estimate on $B$, that is, we obtain $\|\varphi_i\|_{L^p_2(B)}\le C_B$ for any $i$ and $p>0$. Then, the proof of \cite{Tru84} leads to the uniform $C^{2,\alpha}(B)$ estimate $\|\varphi_i\|_{C^{2,\alpha}(B)}\le C_B$. Then, by the classical Schauder estimates (e.g. \cite{GT} Problem 6.1), we obtain the uniform $C^{k,\alpha}(B)$ estimate, that is, we have $\|\varphi_i\|_{C^{k,\alpha}(B)}\le C_B$ for any $k$. Therefore $\varphi_i$ converges to $\varphi$ locally smoothly on ${X}\setminus D$.
\end{proof}
\noindent By definition, the density function of the volume form of any adapted current $T$ with respect to $\omega_0^n$ lies in $L^p(X,\omega_0^n)$ with $p>1$. In particular, the following holds:
\begin{lemm}[refer to Theorem \ref{MA BEGZ}]\label{adapted curr prop}
Let $X$ be a compact K\"{a}hler manifold and $\alpha$ be a nef and big class on $X$ such that $D=E_{nK}(\alpha)$ is a snc divisor. Let $s_D$ be a defining section of $D$ and $h_D$ be a smooth hermitian metric on $\mathcal{O}(D)$ with $|s_D|_{h_D}\le 1$. Then, any adapted current $T$ in $\alpha$ has minimal singularities and it satisfies
$$
|s_D|_{h_D}^k\langle T^n\rangle\le C\omega_0^n
$$
for some large $k>0$. Let $\theta$ be a smooth $(1,1)$-form in $\alpha$ and $\varphi\le 0$ be a $\theta$-plurisubharmonic function on $X$ with $T=\theta+dd^c\varphi$. Then there exists a constant $C>0$ such that the following holds for any $i$:
$$
\varphi\ge \log|s_D|^2-C.
$$
\end{lemm}
\noindent The following lemma will be used later. The proof is obvious.
\begin{lemm}\label{adapted app}
Let $X$ be a compact K\"{a}hler manifold and $\alpha$ be a nef and big class on $X$. Let $\pi:Y\to X$ be a composition of blow-ups along complex submanifolds contained in $E_{nK}(\alpha)$ or of codimension at least 2. Then a closed positive $(1,1)$-current $T$ in $\alpha$ is adapted if and only if $\pi^*T$ in $\pi^*\alpha$ is adapted.
\end{lemm}
Chen-Chiu-Hallgren-Sz${\rm \acute{e}}$kelyhidi-T${\rm \hat{o}}$-Tong \cite{CCHSTT25} proved the following:
\begin{theo}[\cite{CCHSTT25}]\label{KE adapted}
Let $X$ be a compact normal K\"{a}hler variety with log terminal singularities. Assume that $X$ admits a singular K\"{a}hler-Einstein metric $\omega$. Let $\pi:Y\to X$ be a resolution of singularities such that $Y$ is a compact K\"{a}hler manifold. Then $\pi^*\omega$ is an adapted current in the sense of Definition \ref{adapted defi}. 
\end{theo}
\subsection{\texorpdfstring{$T$}{T}-adapted HYM metrics}
Bando-Siu \cite{BS94} introduced the notion of admissible Hermitian-Yang-Mills metrics to establish the Kobayashi-Hitchin correspondence for reflexive sheaves on compact K\"{a}hler manifolds. In this section, we introduce the notion of adapted Hermitian-Yang-Mills metrics by imposing two additional conditions on admissible Hermitian-Yang-Mills metrics.
\begin{defi}\label{adapted HYM defi}
Let $X$ be a compact K\"{a}hler manifold and $\alpha$ be a nef and big class on $X$. Let $T$ be a closed positive $(1,1)$-current in $\alpha$ which is smooth K\"{a}hler on $\Amp(\alpha)$. 
Let $\pi:Y\to X$ be a composition of blow-ups so that $D=E_{nK}(\alpha)$ is a snc divisor. We denote by $s_D$ a defining section of $D$ and $h_D$ a smooth hermitian metric on $\mathcal{O}(D)$ such that $|s_D|_{h_D}\le 1$.
Let $(E, h_0)$ be a complex hermitian vector bundle on $X$ with a smooth hermitian metric $h_0$ on $E$.
\begin{enumerate}
\item A $T$-admissible Hermitian-Yang-Mills connection ($T$-admissible HYM connection in short) on $(E,h_0)$ is a smooth hermitian connection $\nabla$ on $(E,h_0)|_{\Amp(\alpha)}$ which satisfies the following conditions:
\begin{enumerate}
\item (integrability) $\delbar^{\nabla}\circ\delbar^{\nabla}=0$.
\item (HYM equation) $\sqrt{-1}\Lambda_TF_{\nabla}=\lambda\id_E$ on $\Amp(\alpha)$ for some constant $\lambda\in\mathbb{R}$.
\item (admissible condition) $\int_{\Amp(\alpha)}|F_{\nabla}|_{h_0,T}^2T^n<\infty$.
\end{enumerate}
We  call $\lambda$ in (b) the $T$-Hermitian-Yang-Mills constant of $\nabla$.
\item If $E$ admits a holomorphic structure $\delbar^E$, then a $T$-admissible Hermitian-Yang-Mills metric on $(E,\delbar^E)$ is a smooth hermitian metric $h$ on $E|_{\Amp(\alpha)}$ whose Chern connection $\nabla_h$ satisfies the following conditions:
\begin{itemize}
\item[(b)$^{\prime}$] $\sqrt{-1}\Lambda_TF_{\nabla_h}=\lambda\id_E$ on $\Amp(\alpha)$ for some constant $\lambda\in\mathbb{R}$.
\item[(c)$^{\prime}$]  $\int_{\Amp(\alpha)}|F_{\nabla_h}|_{h,T}^2T^n<\infty$.
\end{itemize}
We call $\lambda$ in (b)$^{\prime}$ the $T$-Hermitian-Yang-Mills constant of $h$.
\item If $E$ admits a holomorphic structure $\delbar^E$, then a $T$-adapted Hermitian-Yang-Mills metric on $(E,\delbar^E)$ is a $T$-admissible HYM metric on $(E,\delbar^E)$ which additionally satisfies the following two conditions: If we denote by $h=h_0\Psi^2$ with a positive definite $h_0$-hermitian endomorphism $\Psi$, then it satisfies the following two conditions, called by $T$-adapted conditions.
\begin{enumerate}
\item[{(d)}] There exists $m\in\mathbb{Z}_{\ge 0}$ such that
$$
\sup_{Y\setminus D}|s_D^m\pi^*\Psi|_{h_D^m\otimes \pi^*h_0\otimes \pi^*h_0^*}<\infty.
$$
\item[{(e)}] There exists $l\in\mathbb{Z}_{\ge 0}$ such that
$$
\int_{Y\setminus D}|s_D^l\partial^{\nabla_{\pi^*h_0}}(\pi^*\Psi)|^2_{h_D^l\otimes\pi^*h_0\otimes\pi^*h_0^*,T}(\pi^*T)^n<\infty.
$$
\end{enumerate}
\end{enumerate}
\end{defi}
\noindent The following lemma will be used later. The proof is obvious.
\begin{lemm}\label{inv adapted HYM}
Let $X$ be a compact K\"{a}hler manifold and $\alpha$ be a nef and big class. Let $T$ be a closed positive $(1,1)$-current in $\alpha$ which is smooth K\"{a}hler on $\Amp(\alpha)$. Let $E$ be a holomorphic vector bundle on $X$.  Let $\pi:Y\to X$ be a composition of blow-ups along complex submanifolds contained in $E_{nK}(\alpha)$ or of codimension at least 2. Then a smooth hermitian metric $h$ on $E|_{\Amp(\alpha)}$ is a $T$-adapted HYM metric on $E$ if and only if $\pi^*h$ is a $\pi^*T$-adapted HYM metric on $\pi^*E$.
\end{lemm}

\section{Proof of Theorem \ref{main thm2 intro}}
In this section, we use the following notations.
\begin{notation}\label{notations}
Let $(X,\omega_0)$ be a compact K\"{a}hler manifold with a smooth K\"{a}hler metric $\omega_0$ and $\alpha$ be a nef and big class on $X$. 
\begin{itemize}
\item Suppose that $D=E_{nK}(\alpha)$ is an effective snc divisor. We denote by $s_D$ a defining section of $D$ and by $h_D$ a smooth hermitian metric on the holomorphic line bundle $\mathcal{O}(D)$ such that $|s_D|_{h_D}\le 1$. 
\item Let $T$ be an adapted closed positive $(1,1)$-current in $\alpha$ satisfying $|s_D|^k\langle T^n\rangle\le C\omega_0^n$ (refer to Lemma \ref{adapted curr prop}). Let $m\ge 0$ be a nonnegative integer such that $T\ge |s_D|^{2m}_{h_D}\omega_0$ holds in the sense of currents.
\item Let $\omega_i$ be a smooth K\"{a}hler metric contained in a K\"{a}hler class $\alpha+(1/i)\omega_0$ which gives an adapted approximation of $T$ in the sense of Definition \ref{adapted defi}.
\item Let $(E,h_0)$ be a pair consisting of a holomorphic vector bundle $E$ on $X$ and a smooth hermitian metric $h_0$ on $E$. Assume that $E$ is $\alpha^{n-1}$-slope stable.
\end{itemize}
\end{notation}
\subsection{Preliminary results}

First of all, we prepare the following existence result of cut-off functions.
\begin{lemm}[\cite{EG15}, \S 4.7, Theorem 3]\label{cut-off}
Let $(X,g)$ be a compact Riemannian manifold of dimension $n$ and $Z$ be a compact subset of $X$ of Hausdorff codimension 2. Then, there exists a sequence of smooth nonnegative functions $\eta_k:X\to[0,1]$ satisfying the following conditions:
\begin{enumerate}
\item $\Supp(\eta_k)\Subset X\setminus Z$.
\item $\{\eta_k=1\}\Subset \{\eta_{k+1}= 1\}$ and $\bigcup_{k=1}^{\infty}\{\eta_k=1\}=X\setminus Z$.
\item $\lim_{k\to\infty}\int_X|d\eta_k|_g^2\vol_g=0$.
\end{enumerate}
\end{lemm}

\begin{lemm}[\cite{Jin25-2}]\label{Uhl limit}
In the notations in Notation \ref{notations}, the following holds:
\begin{enumerate}
\item $E$ is $\{\omega_i\}^{n-1}$-slope stable.
\end{enumerate}
Let $h_i=h_0\Psi_i^2$ be the $\omega_i$-HYM metric on $E$ where $\Psi_i$ is a positive definite $h_0$-hermitian endomorphism of $E$ with $\det\Psi_i=1$. Then
\begin{enumerate}
\setcounter{enumi}{1}
\item $\nabla_i:=\Psi_i\circ\nabla_{h_i}\circ\Psi_i^{-1}$ is an $\omega_i$-HYM connection of a complex hermitian vector bundle $(E,h_0)$.
\item $\Psi_i$ defines a holomorphic isomorphism from $(E,\delbar^{E})$ to $(E,\delbar^{\nabla_i})$.
\item There is a constant $C>0$ independent of $i$ such that the $L^2$-norm of the curvature tensor $F_{\nabla_i}$ of $\nabla_i$ is uniformly bounded by $C$, that is, the following holds:
$$
\int_X|F_{\nabla_i}|^2_{h_0,\omega_i}\omega_i^n\le C.
$$
\end{enumerate}
\end{lemm}
\begin{proof}
The first statement (1) is a consequence of Lemma \ref{open stable}. We show (2), (3) and (4).
We firstly prove that $\nabla_i$ is a $h_0$-connection. It is a consequence of the direct computation as follows: For any local section $u$ and $v$ of $E$,
\begin{align*}
d\big(h_0(u,v)\big)
&=d\big(h_i(\Psi_i^{-1} u,\Psi_i^{-1} v)\big)\\
&=h_i\big(\nabla_{h_i}(\Psi_i^{-1} u), \Psi_i^{-1} v\big)+h_0\big(\Psi_i^{-1} u, \nabla_{h_i}(\Psi_i^{-1} v)\big)\\
&=h_0\big((\Psi_i\circ\nabla_{h_i}\circ\Psi_i^{-1})u, v\big)+ h_0\big(u, (\Psi_i\circ\nabla_{h_i}\circ\Psi_i^{-1})v \big)\\
&=h_0(\nabla_i u, v) + h_0(u,\nabla_i v).
\end{align*}
Next we see that $\Psi_i:(E,\delbar^E)\to (F,\delbar^{\nabla_i})$ is holomorphic. Since $\delbar^{\nabla_{h_i}}=\delbar^E$, the following computation works by the definition of $\nabla_i$:
$$
(\delbar^{\nabla_i}\otimes\delbar^{E^*})\Psi_i
=\delbar^{\nabla_i}\circ\Psi_i-\Psi_i\circ\delbar^E
=\Psi_i\circ\delbar^{\nabla_{h_i}}-\Psi_i\circ\delbar^E=0.
$$
We recall that for any endomorphism $A$ of $E$, we have at each point
$$
A^{*_{h_i}}=h_i^{-1}\overline{A}^{t}h_i=\Psi_i^{-2}(h_0^{-1}\overline{A}^{t}h_0)\Psi_i^2=\Psi_i^{-2}A^{*_{h_0}}\Psi_i^2.
$$
Hence 
$$
|A|_{h_i}^2=\Tr(A\cdot A^{*_{h_i}})=\Tr((\Psi_i A \Psi_i^{-1})(\Psi_i^{-1} A^{*_{h_0}} \Psi_i)).
$$
Then, since $F_{\nabla_i}=\Psi_i\circ F_{h_i}\circ \Psi_i^{-1}$, we have that
$$
|F_{\nabla_i}|_{h_0}^2=\Tr((\Psi_i\circ \sqrt{-1}F_{h_i}\circ\Psi_i^{-1})(\Psi_i^{-1}\circ\sqrt{-1}F_{h_i}^{*_{h_0}}\circ\Psi_i))=|F_{h_i}|_{h_i}^2.
$$ 
Hence, since $h_i$ is $\omega_i$-HYM, we obtain that $\nabla_i$ is an $\omega_i$-HYM connection of the complex hermitian vector bundle $(E, h_0)$. The uniform $L^2$-bound in (4) follows from the following standard inequality (refer to \cite[Theorem 4.4.7]{Kob87}):
\begin{align*}
(2rc_2(E)-(r-1)c_1(E)^2)\cdot\{\omega_i\}^{n-2}
&=c_n\int_X|F_{\nabla_{h_i}}|^2_{h_0}\omega_i^n-c_n\int_X|\Lambda_{\omega_i}F_{\nabla_{i}}|_{h_0}^2\omega_i^n\\
&\ge c_n\int_X|F_{\nabla_{h_i}}|^2_{h_0}\omega_i^n-C,
\end{align*}
where the last inequality follows from the $\omega_i$-HYM equation $\Lambda_{\omega_i}F_{\nabla_i}=\lambda_i\id$ and 
$$
\lambda_i=\frac{1}{\omega_i^n}\frac{c_1(E)\cdot\{\omega_i\}^{n-1}}{\rk E}.
$$
\end{proof}
\subsection{Proof of Theorem \ref{main thm2 intro} (1) \texorpdfstring{$\Rightarrow$}{Rightarrow} (2)}
In the this subsections, we construct $\Psi_{\infty}^{\pm1}$, the nontrivial limit of $\Psi_i^{\pm1}$ in Lemma \ref{Uhl limit}, which gives a $T$-adapted HYM metric on $E$ (refer to Lemma \ref{uniform est defi lemm}, Lemma \ref{C0 to L2 limit}, Lemma \ref{C0 to L1}, Lemma \ref{L1 est} and Theorem \ref{KH corr nef big with esti}).
We will use the same notations in Lemma \ref{Uhl limit}.
The following mean value type inequality by Guo-Phong-Sturm \cite{GPS24} plays an important role in this subsection.
\begin{lemm}[\cite{GPS24}, Lemma 2]\label{mv ineq}
Let $\omega_i$ be a smooth K\"{a}hler metric in Notation \ref{notations}. Suppose that $v\in L^1(X, \omega_i^n)$ is a function which satisfies $\int_Xv\omega_i^n=0$ and
$$
v\in C^2(\overline{\Omega_0}), \hspace{4mm} \Delta_{\omega_i}v\ge -a\hspace{2mm} \hbox{on $\Omega_0$}
$$
for some constant $a>0$ and $\Omega_s=\{v> s\}$. Then, there exists a constant $C>0$ independent of $i$ and $v$ such that 
$$
\sup_Xv\le C\big(a+\|v\|_{L^1(X,\omega_i^n)}\big).
$$
\end{lemm}
\subsubsection{Reduction of the proof to a uniform \texorpdfstring{$C^0$}{C0}-estimate}
\begin{lemm}\label{uniform est defi lemm}
We use notations in Notation \ref{notations}.
If there exists a constant $C>0$ independent of $i$ such that
$$
\sup_X|s_D\Psi_i|_{h_D\otimes h_0\otimes h_0^*}\le C
$$
holds for any $i$. Then
there exists a constant $C>0$ such that the following inequality holds for any $i$:
\begin{enumerate}
\item $\sup_X|s_D\Psi_i^{-1}|_{h_0\otimes h_0^*\otimes h_D}\le C$.
\item $\int_X|s_D^{m+1}\partial^{\nabla_{h_0}}{\Psi_i}^{\pm1}|^2_{h_0\otimes h_0^*\otimes h_{D}^{m+1}, \omega_i}\omega_i^n\le C.$
\end{enumerate}
\end{lemm}
\begin{proof}
(1) is a direct consequence of $\det \Psi_i=1$. We prove (2).
Let $h_D$ be a smooth hermitian metric on $\mathcal{O}(D)$. We denote by $\nabla_D:=\nabla_{i}\otimes\nabla_{h_0}^*\otimes\nabla_{h_D}$ a $h_0^*\otimes h_0\otimes h_D$-connection on $\Hom(E,E)\otimes \mathcal{O}(D)$.
If we use the identity $\Box^{\partial^{\nabla_D}}_{\omega_i}=\Box^{\delbar^{\nabla_D}}_{\omega_i}-\sqrt{-1}[\Lambda_{\omega_i},F_{\nabla_D}]$, and $T\ge |s_D|^{2m}\omega_0$, the following calculation works:
\begin{align}\label{estimate on difference}
&\int_X\langle \nabla_D(s_D^{m+1}{\Psi_i}),\nabla_D(s_D^{m+1}{\Psi_i})\rangle_{h_0^*\otimes h_0\otimes h_D,\omega_i}\omega_i^n\notag\\
&=\int_X\langle\Delta^{\nabla_D}_{\omega_i}(s_D^{m+1}{\Psi_i}),s_D^{m+1}{\Psi_i}\rangle_{h_0^*\otimes h_0\otimes h_D}\omega_i^n \notag\\
&=-\int_X\langle\sqrt{-1}\Lambda_{\omega_i}F_{\nabla_D}\cdot (s_D^{m+1}{\Psi_i}), s_D^{m+1}{\Psi_i}\rangle_{h_0^*\otimes h_0\otimes h_D}\omega_i^n \notag\\
&=\int_X\langle-\sqrt{-1}\Lambda_{\omega_i}F_{\nabla_{h_0}\otimes\nabla_{h_D}}\circ(|s_D|^{2m}s_D{\Psi_i}),s_D{\Psi}\rangle_{h_0^*\otimes h_0\otimes h_D}\omega_i^n\notag\\
&\hspace{4mm}+\int_X\langle(|s_D|^{2m}s_D{\Psi_i})\circ(\sqrt{-1}\Lambda_{\omega_i}F_{\nabla_i}),s_D{\Psi_i}\rangle_{h_0^*\otimes h_0\otimes h_D}\omega_i^n\notag\\
&\le C\int_X\langle|s_D|^{2m}\sqrt{-1}\Lambda_{\omega_i}\omega_0\cdot(s_D{\Psi_i}),s_D{\Psi_i}\rangle_{h_0^*\otimes h_0\otimes h_D}\omega_i^n\notag\\
&\hspace{4mm}+|\lambda_i|\int_X|s_D{\Psi_i}|_{h_0^*\otimes h_0\otimes h_D}^2\omega_i^n\notag\\
&\le C.
\end{align}
Since $\sup_{X\setminus D}|s_D\Psi_i|\le C$ by the assumption, we obtain that 
$$
\int_X|s_D^{m+1}(\nabla_i\otimes\nabla_{h_0}^*)\Psi_i|^2\omega_i^n\le C.
$$
Since $\nabla_i=\Psi_i\circ\nabla_{h_i}\circ\Psi_i^{-1}$ by Lemma \ref{Uhl limit}, the result follows by the follwing computation:
\begin{align*}
(\nabla_i\otimes\nabla_{h_0}^*)\Psi_i
&=\nabla_i\circ\Psi_i-\Psi_i\circ\nabla_{h_0}\\
&=\Psi_i\circ(\nabla_{h_i}-\nabla_{h_0})\\
&=\Psi_i\circ(\Psi_i^{-2}\partial^{\nabla_{h_0}}(\Psi_i^2))\\
&=2\partial^{\nabla_{h_0}}\Psi_i.
\end{align*}
The statement for $\Psi_i^{-1}$ is proved by applying the same argument to the holomorphic morphism $\Psi_i^{-1}:(E,\delbar^{\nabla_i})\to (E,\delbar^E)$.
\end{proof}
The following lemma shows that it suffices to derive the uniform $C^0$-estimates for the construction of $\Psi_{\infty}^{\pm1}$.
\begin{lemm}\label{C0 to L2 limit}
We use the notations in Notation \ref{notations} and Lemma \ref{Uhl limit}. If there exists a constant $C>0$ independent of $i$ such that 
\begin{equation}\label{uniform esti diff lemm ass}
\sup_X|s_D\Psi_i|_{h_D\otimes h_0\otimes h_0^*}\le C
\end{equation}
holds for any $i$, then there exists a smooth $h_0$-hermitian endomorphism $\Psi_{\infty}$ of $E|_{X\setminus D}$ such that the following holds:
\begin{enumerate}
\item For any $0<\alpha<1$, there exists a subsequence of $\{\Psi_i\}_i$, also denoted by $\{\Psi_i\}_i$, such that for any compact subset $A\subset X\setminus D$, 
$$
\lim_{i\to\infty}\|\Psi_i-\Psi_{\infty}\|_{C^{1,\alpha}(A)}=0,
$$
where the $C^{1,\alpha}$-norm is defined by the reference smooth K\"{a}hler metric $\omega_0$. 
\item There exists a constant $C>0$ such that 
\begin{itemize}
\item $\sup_{X\setminus D}|s_D\Psi_{\infty}^{\pm1}|_{h_D\otimes h_0\otimes h_0^*}\le C.$
\item $\int_{X\setminus D}|s_D^{m+1}\partial^{\nabla_{h_0}}\Psi_{\infty}^{\pm1}|^2_{h_D\otimes h_0\otimes h_0^*, T}T^n\le C.$
\end{itemize}
\item Furthermore, $\Psi_{\infty}$ is nontrivial: $\Psi_{\infty}\not\equiv 0$.
\end{enumerate}
\end{lemm}
\begin{proof}
It is a direct consequence of \cite[Proposition 1]{BS94}. In fact, the $\omega_i$-HYM equation is locally written as
\begin{equation}\label{uniform esti diff lemm eq1}
\sqrt{-1}\Lambda_{\omega_i}\delbar(\partial h_i\cdot h_i^{-1})=\lambda_i\id.
\end{equation}
By Lemma \ref{uniform est defi lemm}, we have
\begin{equation}\label{uniform esti diff eq2}
\sup_{X\setminus D}|s_D\Psi_i^{-1}|\le C.
\end{equation}
By (\ref{uniform esti diff lemm ass}), (\ref{uniform esti diff lemm eq1}) and (\ref{uniform esti diff eq2}), \cite[Proposition 1]{BS94} implies that for any relatively compact open subset $K\subset X\setminus D$ and $0<\alpha\rq{}<1$, there exists a constant $C_K>0$ such that the following holds:
$$
\|\Psi_i\|_{C^{1,\alpha\rq{}}(K)}\le C_K.
$$
Here the $C^{1,\alpha\rq{}}$-norm is defined by $\omega_0$. Since the natural inclusion $C^{1,\alpha\rq{}}(\overline{K})\hookrightarrow C^{1,\alpha}(\overline{K})$ is compact for any $0<\alpha<\alpha\rq{}<1$, we can find a subsequence of $\{\Psi_i\}_i$ such that $\Psi_i\to \Psi_{\infty, K}$ in $C^{1,\alpha}(K)$.  Then, if we consider a sequence of relatively compact open subsets by $X\setminus D_{\varepsilon_k}$ for $\varepsilon_k\searrow 0$ where $D_{\varepsilon_k}$ is the $\varepsilon_k$-neighborhood of $D$ with respect to $\omega_0$, then the standard diagonal procedure gives a subsequence of $\{\Psi_i\}_i$ such that $\Psi_i\to \Psi_{\infty}$ in $C^{1,\alpha}(K)$ for any relatively compact open subset $K\subset X\setminus D$. Let us define a hermitian metric on $E|_{X\setminus D}$ by $h_{\infty}:=h_0\Psi_{\infty}^2$. Then it satisfies
\begin{equation}\label{uniform esti diff eq3}
\sqrt{-1}\Lambda_T\delbar(\partial h_{\infty}\cdot h_{\infty}^{-1})=\lambda\id
\end{equation}
where $\lambda=\frac{1}{\alpha^n}\frac{c_1(E)\cdot \alpha^{n-1}}{\rk E}$. Since $h_{\infty}^{\pm1}$ is in $C^{1,\alpha}(K)\cup L^2_1(K)$ for any $K\Subset X\setminus D$, the standard elliptic estimate concludes that $\Psi_{\infty}$ is smooth on $X\setminus D$. The remaining statements $(2)$ and $(3)$ are the direct consequences of Lemma \ref{uniform est defi lemm}, (\ref{uniform esti diff lemm ass}) and $\det\Psi_i=1$.
\end{proof}

\subsubsection{Reduction of the uniform \texorpdfstring{$C^0$}{C0}-estimate to a uniform \texorpdfstring{$L^1$}{L1}-estimate}
By the next lemma, we can reduce the uniform $C^0$-estimate to the uniform $L^1$-estimate.
\begin{lemm}\label{C0 to L1}
We use the notations in Notation \ref{notations} and Lemma \ref{Uhl limit}. We denote by $\Psi_i=e^{s_i}$ for some $h_0$-hermitian endomorphism $s_i$ with $\Tr(s_i)=0$.
Then there exists a constant $C>0$ independent of $i$ such that the following holds:
\begin{equation}\label{main theo eq1}
\sup_X\left(\log|s_D|_{h_D}^2+|s_i|_{h_0}\right)
\le \sup_X\log|s_D\Psi_i^{\pm1}|_{h_0\otimes h_D}^2
\le C(1+\int_X\left|\log|s_D|^2+|s_i|_{h_0}\right|\omega_i^n).
\end{equation}
\end{lemm}
\begin{proof}
Since $\det\Psi_i\equiv 1$, we have $|\Psi_i|_{h_0}\ge 1$. Hence there exists a constant $a>0$ such that 
$$
\Omega_a=\{\log|s_D\Psi_i|^2>-a\}
$$
is not empty for all $i$ and $\Omega_a$ is a relatively compact open subset in $X\setminus D$ since $|\Psi_i|\ge 1$ on $X$ and $|s_D|=0$ on $D$. Hence $\omega_i$ smoothly converges to $T$ on $\Omega_a$ and $T$ is smooth K\"{a}hler metric on $\Omega_a$. Then, as \cite[Lemma 3.1]{Sim88}, there exists a constant $C>0$ independent of $i$ such that
$$
\Delta_{\omega_i}\left(\log|s_D\Psi_i|^2+a-\int_X\log|s_D\Psi_i|^2\omega_i^n\right)
\ge -2|\Lambda_{\omega_i}F_{h_0}|_{h_0}-2|\Lambda_{\omega_i}F_{h_D}|_{h_D}-C_1
\ge -C
$$
holds on $\Omega_a$. In fact, let us consider the following composition of holomorphic morphisms:
$$
(E,\delbar^E)\xrightarrow{\Psi_i} (E,\delbar^{\nabla_i})\xrightarrow{s_D}(E\otimes\mathcal{O}(D), \delbar^{\nabla_i\otimes\nabla_{h_D}}).
$$
Then the following computation works on $\Omega_a\Subset X\setminus D$:
\begin{align*}
\Delta_{\omega_i}\log|s_D\Psi_i|^2_{h_0\otimes h_0^*\otimes h_D}
&=2\sqrt{-1}\Lambda_{\omega_i}\partial\left(|s_D\Psi_i|^{-2}\langle s_D\Psi_i,\partial^{\nabla_i\otimes\nabla_{h_0}^*\otimes\nabla_{h_D}}(s_D\Psi_i)\rangle \right)\\
&=-2|s_D\Psi_i|^{-4}\left|\langle s_D\Psi_i,\partial^{\nabla_i\otimes\nabla_{h_0}^*\otimes\nabla_{h_D}}(s_D\Psi_i)\rangle \right|^2\\
&\hspace{4mm}+2|s_D\Psi_i|^{-2}\left|\partial^{\nabla_i\otimes\nabla_{h_0}^*\otimes\nabla_{h_D}}(s_D\Psi_i)\right|^2\\
&\hspace{4mm}-2|s_D\Psi_i|^{-2}\langle s_D\Psi_i, \sqrt{-1}\Lambda_{\omega_i}F_{\nabla_i\otimes\nabla_{h_0}^*\otimes\nabla_{h_D}}(s_D\Psi_i)\rangle\\
&\ge -2|\lambda_i|-2|\Lambda_{\omega_i}F_{h_0}|-2|\Lambda_{\omega_i}F_{h_D}|.
\end{align*}
Then, by the mean value type inequality of Guo-Phong-Sturm (Lemma \ref{mv ineq}), we obtain that there is a positive constant $C>0$ independent of $i$ such that
\begin{align*}
&\sup_X\left(\log|s_D\Psi_i|^2+a-\int_X\log|s_D\Psi_i|^2\omega_i^n\right)\notag\\
&\le C\left(1+\int_X\left|\log|s_D\Psi_i|^2+a-\int_X\log|s_D\Psi_i|^2\omega_i^n\right|\omega_i^n\right)
\end{align*}
holds. Since the constant $a>0$ is independent of $i$, we can find a constant $C>0$ independent of $i$ such that the following holds:
\begin{equation}\label{C0 to L1 eq1}
\sup_X\left(\log|s_D|^2+\log|\Psi_i|^2\right)\le C\big(1+\int_X|\log\left|s_D\Psi_i|\right|\omega_i^n\big).
\end{equation}
By applying the same argument to $\log|s_D\Psi^{-1}_i|^2$, we obtain
\begin{equation}\label{C0 to L1 eq2}
\sup_X\left(\log|s_D|^2+\log|\Psi_i^{-1}|^2\right)\le C\big(1+\int_X|\log\left|s_D\Psi_i^{-1}|\right|\omega_i^n\big),
\end{equation}
We denote by $\Psi_i=e^{s_i}$ for some $h_0$-hermitian endomorphism $s_i$ of $E$. Choosing a suitable $h_0$-orthonormal basis $e_1,\ldots, e_r$ around $x\in X$ so that $s_i$ becomes a diagonal matrix. We denote by $s_i=\lambda_1e_1+\cdots+\lambda_re_r$. Then, for any $i$, we have
\begin{equation}\label{C0 to L1 eq3}
\log|\Psi_i^{\pm1}|^2
=\log(e^{\pm2\lambda_1}+\cdots+e^{\pm2\lambda_r})\ge \pm2\lambda_i
\end{equation}
and 
\begin{align}\label{C0 to L1 eq4}
\log|\Psi_i^{\pm1}|^2
&\le \log(e^{2|\lambda_1|}+\cdots+e^{2|\lambda_r|})\notag\\
&\le \log(e^{2|\lambda_1|+\cdots+2|\lambda_r|}+C)\notag\\
&\le C(|\lambda_1|+\cdots+|\lambda_r|+1)\notag\\
&\le C(|s_i|_{h_0}+1)
\end{align}
Then, by (\ref{C0 to L1 eq1}), (\ref{C0 to L1 eq2}), (\ref{C0 to L1 eq3}) and (\ref{C0 to L1 eq4}), we obtain that
$$
\sup_X\left(\log|s_D|^2+|s_i|^2\right)
\le \sup_X\log|s_D\Psi_i^{\pm1}|_{h_0\otimes h_D}^2
\le C(1+\int_X\left|\log|s_D|^2+|s_i|_{h_0}\right|\omega_i^n).
$$
\end{proof}
\subsubsection{Proof of the uniform \texorpdfstring{$L^1$}{L1}-estimate}
Then we establish the uniform $L^1$-estimate.
\begin{lemm}\label{L1 est}
We use the notations in Notation \ref{notations}, Lemma \ref{Uhl limit} and Lemma \ref{C0 to L1}.
Suppose that the $L^1$-norm in Lemma \ref{C0 to L1} diverges to $\infty$, that is, we suppose that $\delta_i^{-1}:=\int_X|\log|s_D|+|s_i||\omega_i^n\nearrow\infty$ in $i\to\infty$. Then, there exists a nontrivial $h_0$-hermitian endomorphism $u_{\infty}$ which satisfies the following.
\begin{enumerate}
\item For any relatively compact open subset $K\Subset X\setminus D$, there exists a subsequence of $\{u_i\}_i$, also denoted by $\{u_i\}_i$, such that $u_i:=\delta_is_i\to u_{\infty}$ weakly in $L^2_1(K, T)$ and strongly in $L^2(K, T^n)$.
\item Let $\Psi:\mathbb{R}\times \mathbb{R}\to \mathbb{R}_{>0}$ be a smooth positive function such that $\Psi(a,b)<\frac{1}{b-a}$ holds for any $a<b$. Then we have
$$
0\ge \int_{X\setminus D}\langle\Psi(u_{\infty})(\partial^{\nabla_{h_0}}u_{\infty}),\partial^{\nabla_{h_0}}u_{\infty} \rangle_{h_0,T}T^n+\sqrt{-1}\int_{X\setminus D}\Tr(u_{\infty}\Lambda_TF_{h_0})T^n.
$$
We can refer to \cite[Section 4]{Sim88} for the definition of the first term above (see also \cite[section 2.C]{KO25}).
\item The eigen values of $u_{\infty}$ are constants.
Let us denote by $\lambda_1,\ldots, \lambda_k$ the eigen values of $u_{\infty}$ and suppose that the multiplicity of $\lambda_j$ is $i_j$ where $i_1+\cdots+i_k=r$.
Then, there exists $l\le r$ such that $\mathcal{F}_k:=\Ker(u_{\infty}-\Sigma_{j=1}^l\lambda_jI_{i_j})$ produces a coherent subsheaf of $E$ on $X$ and it satisfies
$$
\mu_{\alpha}(\mathcal{F}_k)\ge \mu_{\alpha}(E).
$$
\end{enumerate}
In particular, there exists a uniform constant $C>0$ such that the following holds:
$$
\int_X\left(-\log|s_D|^2+|s_i|\right)\omega_i^n\le C.
$$
\end{lemm}
\begin{proof}
The proof of this lemma is inspired by the argument of \cite[Claim 4.6]{KO25}. We will divide the proof into four-steps as \cite[Claim 4.6]{KO25}. 

\textbf{1st Step}: This is the only step that requires an argument different from that in \cite[Claim 4.6]{KO25}. Let $\Phi(x,y):=\frac{\exp(x-y)-1}{x-y}$ for $x,y\in\mathbb{R}$. Then, since $h_i=h_0e^{s_i}$ defines an $\omega_i$-HYM metric on $E$ and $u_i=\delta_is_i$, we have (refer to \cite[Lemma 2.9]{KO25})
\begin{equation}\label{L1 est eq0}
0=\delta_i^{-1}\int_X\langle\Phi(u_i/\delta_i)(\partial^{\nabla_{h_0}}u_i),\partial^{\nabla_{h_0}}u_i\rangle_{h_0,\omega_i}\omega_i^n+\int_X\Tr(u_i\sqrt{-1}\Lambda_iF_{h_0})\omega_i^n.
\end{equation}
We then prove $\int_X\Tr(u_i\sqrt{-1}\Lambda_{\omega_i}F_{h_0})\omega_i^n$ is uniformly bounded from above and below. Since $h_0$ is a smooth hermitian metric on $E$ and $\omega_0$ is a smooth K\"{a}hler metric on $X$, there is a constant $C>0$ such that
$$
-C\omega_0\cdot \Id_E\le \sqrt{-1}F_{h_0}\le C\omega_0\cdot \Id_E.
$$
By Lemma \ref{C0 to L1} and the definition of $u_i$, we have
\begin{equation}\label{L1 est eq1}
|u_i|_{h_0}\le C-C\delta_i\log|s_D|_{h_D}^2.
\end{equation}
Then we have
\begin{equation}\label{L1 est eq1.55}
|u_i|_{h_0}|\sqrt{-1}\Lambda_{\omega_i}F_{h_0}|_{h_0}\omega_i^n\le C(C-\delta_i\log|s_D|^2)\omega_0\wedge\omega_i^{n-1}.
\end{equation}
By Claim \ref{log int nef big} below, we obtain
\begin{equation}\label{L1 est eq1.5}
\left|\delta_i^{-1}\int_X\langle\Phi(u_i/\delta_i)(\partial^{\nabla_{h_0}}u_i),\partial^{\nabla_{h_0}}u_i\rangle_{h_0,\omega_i}\omega_i^n\right|
\le \int_X|u_i|_{h_0}|\sqrt{-1}\Lambda_{\omega_i}F_{h_0}|_{h_0}\omega_i^n\le C.
\end{equation}
\begin{claim}\label{log int nef big}
There exists a constant $C>0$ such that the following holds for any $i$:
$$
\int_X(-\log|s_D|_{h_D}^2)\wedge\omega_0\wedge\omega_i^{n-1}\le C.
$$
\end{claim}
\begin{proof}[proof of claim]
Let $\theta$ be a smooth $(1,1)$-form in $\alpha$ and $\varphi_i\le 0$ be a $(\theta+(1/i)\omega_0)$-plurisubharmonic function such that $\omega_i=\theta_i+dd^c\varphi_i$ where $\theta_i=\theta+(1/i)\omega_0$. We denote by $\omega_D$ the curvature form of $h_D$.
Then
\begin{align}\label{log int nef big eq1}
&\int_X(-\log|s_{D}|^2)\wedge\omega_0\wedge\omega_i^{n-1}\notag\\
&=\int_X(-\log|s_{D}|^2)\wedge\omega_0\wedge\theta_i^{n-1}
+\Sigma_{p=1}^{n-1}c_p\int_X(-\log|s_{D}|^2)\wedge\omega_0\wedge \theta_i^{n-1-p}\wedge(dd^c\varphi_i)^p
\end{align}
where $c_p>0$ is a constant depending only on $p$.
Since $\theta_i^{n-1}\le C\omega_0^{n-1}$, we have
\begin{equation}\label{log int nef big eq2}
\int_X(-\log|s_{D}|^2)\wedge\omega_0\wedge\theta_i^{n-1}
\le C\int_X(-\log|s_{D}|^2)\wedge\omega_0^n\le C.
\end{equation}
We then prove the boundedness of the second term in RHS of (\ref{log int nef big eq1}) by induction on $p$. Assume $p=1$. Let $a>0$ be any positive constant. Then
\begin{align}\label{log int nef big eq3}
&\int_X(-\max{\{\log|s_{D}|^2, -a\}})\wedge\omega_0\wedge\theta_i^{n-2}\wedge dd^c\varphi_i\notag\\
&=\int_X\varphi_i(-dd^c\max{\{\log|s_{D}|^2, -a\}})\wedge\omega_0\wedge\theta_i^{n-2}\notag\\
&= \int_X\varphi_i\cdot\omega_{D}\wedge\omega_0\wedge\theta_i^{n-2}\notag\\
&\le C\int_X(-\log|s_D|^2)\omega_0^n\notag\\
&\le C
\end{align}
In the third line, we used $\varphi_i\omega_{D}\wedge\theta_i^{n-2}\le -C\log|s_D|^2\omega_0^{n-1}$, which is proved by $\log|s_D|^2-C\le\varphi_i\le0$ (see Lemma \ref{adapted curr prop}) and $-C\omega_0^{n-1}\le \omega_{D}\wedge\theta_i^{n-2}\le C\omega_0^{n-1}.$ 
Then the induction on $p$ works. We only prove the case $p=2$. 
\begin{align}\label{log int nef big eq4}
&\int_X(-\max{\{\log|s_{D}|^2, -a\}})\wedge\omega_0\wedge\theta_i^{n-3}\wedge (dd^c\varphi_i)^{2}\notag\\
&=\int_X\varphi_i(-dd^c\max{\{\log|s_{D}|^2, -a\}})\wedge\omega_0\wedge\theta_i^{n-3}\wedge (dd^c\varphi_i)\notag\\
&=\int_{(\log|s_D|^2>-a)}\varphi_i\cdot\omega_D\wedge\omega_0\wedge\theta_i^{n-3}\wedge(\theta_i+dd^c\varphi_i)
 -\int_{(\log|s_D|^2>-a)}\varphi_i\cdot\omega_D\wedge\omega_0\wedge\theta_i^{n-2}\notag\\
&\le C\int_X(-\max{\{\log|s_D|^2,-a\}})\omega_0^{n-1}\wedge(\theta_i+dd^c\varphi_i)
+C\int_X(-\log|s_D|^2)\omega_0^n\notag\\
&\le C
\end{align}
In the third line, we used $\varphi_i\omega_{D}\wedge\theta_i^{n-2}\le -C\log|s_D|^2\omega_0^{n-1}$ as before and $\theta_i+dd^c\varphi_i\ge 0$. And the fourth line is the result of the case $p=1$. We end the proof of Claim \ref{log int nef big}.
\end{proof}
Since the rest of the proof of (1) and (2) is the same with \cite[Claim 4.6, Step 2, Step 3, Step 4]{KO25}, we only describe the outline.

\textbf{2nd Step}: In this step, we show that $u_i$ subsequencially weakly converges to $u_{\infty}$ in $L^2_{1,\loc}(X\setminus D, T)$-topology.
By (\ref{L1 est eq1}), the eigenvalues of $u_i$ on a relatively compact open subset $K\Subset X\setminus D$ is contained in an interval $[\alpha, \beta]$ for any $i$. We recall that $\Phi(x,y)=\frac{\exc(x-y)-1}{x-y}$ defined in the 1st Step satisfies $\Phi(x,y)\ge \frac{\exc(\alpha-\beta)-1}{\alpha-\beta}$ for any $x, y\in [\alpha,\beta]$ (refer to \cite[Lemma 2.7]{KO25}). Then we have
$$
\langle \Phi(u_i)\partial^{\nabla_{h_0}}u_i,\partial^{\nabla_{h_0}}u_i\rangle_{h_0,\omega_i}
\ge C_K^{-1}|\partial^{\nabla_{h_0}}u_i|_{h_0,\omega_i}^2 \hbox{ on $K$}.
$$
By $\delta_i^{-1}=\int_X\left|\log|s_D|^2+|s_i|\right|\omega_i^n\ge 1$ and (\ref{L1 est eq1.5}), we obtain 
\begin{equation}\label{L1 est eq2}
\int_K|\partial^{\nabla_{h_0}}u_i|_{h_0,\omega_i}^2\omega_i^n
\le C_K\delta_i^{-1}\int_X\langle\Phi(\delta_i^{-1}u_i)\partial^{\nabla_{h_0}}u_i,\partial^{\nabla_{h_0}}u_i\rangle_{h_0,\omega_i}\omega_i^n\le C_K.
\end{equation}
Let us consider a sequence of relatively compact open subsets in $X\setminus D$ by $X\setminus D_{\varepsilon_i}$ where $D_{\varepsilon_i}$ is an $\varepsilon_i$-neighborhood of $D$ with respect to $\omega_0$ with $\varepsilon_i\searrow 0$. We recall that $\omega_i$ smoothly converges to $T$ on each $X\setminus D_{\varepsilon_i}$ and $T$ is a smooth K\"{a}hler metric on $X\setminus D$. Then by (\ref{L1 est eq1}) and (\ref{L1 est eq2}), the standard diagonal procedure gives a subsequence of $\{u_i\}_i$, also denoted by $\{u_i\}_i$, such that 
\begin{equation}\label{L1 est eq2.5}
u_i\to u_{\infty} \hbox{ weakly in $L^2_{1,\loc}(X\setminus D,T)$ and strongly in $L^2_{\loc}(X\setminus D, T^n)$}.
\end{equation}
By (\ref{L1 est eq1.55}) and Claim \ref{log int nef big}, the dominant convergence theorem gives
\begin{equation}\label{L1 est eq3}
\lim_{i\to\infty}\sqrt{-1}\int_X\Tr(u_i\Lambda_{\omega_i}F_{h_0})\omega_i^n\
=\sqrt{-1}\int_{X\setminus D}\Tr(u_{\infty}\Lambda_TF_{h_0})T^n.
\end{equation}
\textbf{3rd Step}: In this step, we show that $u_{\infty}$, the weak limit of $u_i$ constructed in 2nd step, is not identically 0. 
Since $\delta_i=\left(\int_X(-\log|s_D|^2+|s_i|)\omega_i^n\right)^{-1}$ and $u_i=\delta_is_i$, we have
$$
-\int_X\delta_i\log|s_D|^2\omega_i^n+\int_X|u_i|\omega_i^n=1.
$$
If we recall that $\int_X(-\log|s_D|)\omega_i^n\le C$, then we obtain
\begin{equation}\label{L1 est eq4}
2\int_X(-\delta_i\log|s_D|^2)\omega_i^n+\int_X|u_i|\omega_i^n
\ge1-\delta_iC.
\end{equation}
By (\ref{L1 est eq1}) and again by $\int_X(-\log|s_D|)\omega_i^n\le C$, if we take the limit as $i\to \infty$, the dominant convergence theorem allows us to interchange the limit and the integral in the LHS of (\ref{L1 est eq4}). Then we obtain
$$
\int_{X\setminus D}|u_{\infty}|T^n\ge 1.
$$
Hence, if $u_{\infty}=0$, then it implies $0\ge 1$ which is a contradiction. Thus $u_{\infty}\ne 0$.

\textbf{4th Step}: In this step, we show the inequality in (2) in the statement. Let $\Psi:\mathbb{R}\times\mathbb{R}\to \mathbb{R}_{>0}$ be a smooth positive function such that $\Psi(a,b)<\frac{1}{b-a}$ when $a<b$. We fix $K\subset X\setminus D$ a relatively compact open subset. 
We recall that $\delta_i^{-1}\Phi(\delta_i^{-1}a,\delta_i^{-1}b)$ tends to $\frac{1}{b-a}$ when $a<b$ and to $\infty$ when $b\le a$ (refer to \cite[Lemma 2.7]{KO25}). We also recall that, by (\ref{L1 est eq1}), the eigenvalues of $u_i$ lies in an interval $[\alpha,\beta]$ for all $i$. Hence, for sufficiently large $i>0$, we have
$$
\delta_i^{-1}\Phi(\delta_i^{-1}a,\delta_i^{-1}b)\ge \Psi(a,b) \hbox{ for all $a,b\in[\alpha,\beta]$}.
$$
Together with (\ref{L1 est eq0}), we obtain 
\begin{equation}\label{L1 est eq5}
0\ge \sqrt{-1}\int_X\Tr(u_i\Lambda_{\omega_i}F_{h_0})\omega_i^n
+\int_K\langle\Psi(u_i)\partial^{\nabla_{h_0}}u_i,\partial^{\nabla_{h_0}}u_i\rangle_{h_0,\omega_i}\omega_i^n.
\end{equation}
By (\ref{L1 est eq1}) and (\ref{L1 est eq2.5}), we can see that $\Psi^{\frac{1}{2}}(u_i)\to \Psi^{\frac{1}{2}}(u_{\infty})$ in the operator norm topology on $\Hom(L^2(K, T^n), L^q(K, T^n))$ for any $1\le q<2$ (\cite[Proposition 4.1]{Sim88}, see also \cite[Lemma 2.6]{KO25}). Thus, by (\ref{L1 est eq3}) and (\ref{L1 est eq5}), we obtain that for any $\delta>0$, there exists $i_0>0$ such that for any $i>i_0$ the following holds:
\begin{equation}\label{L1 est eq6}
2\delta\ge \|\Psi^{\frac{1}{2}}(u_{\infty})\partial^{\nabla_{h_0}}\|_{L^q(K,T)}^2
+\sqrt{-1}\int_{X\setminus D}\Tr(u_{\infty}\Lambda_TF_{h_0})T^n
\end{equation}
Since $\partial^{\nabla_{h_0}}u_i\to\partial^{\nabla_{h_0}}u_{\infty}$ weakly on $K$ (refer to (\ref{L1 est eq2.5})), we can again use \cite[Proposition 4.1]{Sim88} (see also \cite[Lemma 2.6]{KO25}) to obtain that $\Psi(u_{\infty})\partial^{\nabla_{h_0}}u_i \to \Psi(u_{\infty})\partial^{\nabla_{h_0}}u_{\infty}$ weakly in $L^q(K, T)$. Then, by (\ref{L1 est eq6}), we obtain
\begin{equation}\label{L1 est eq7}
2\delta\ge \|\Psi^{\frac{1}{2}}(u_{\infty})\partial^{\nabla_{h_0}}u_{\infty}\|_{L^q(K,T)}^2+\sqrt{-1}\int_{X\setminus D}\Tr(u_{\infty}\Lambda_TF_{h_0})T^n.
\end{equation}
Here we recall that $1\le q<2$, $K\Subset X\setminus D$ and $\delta>0$ are chosen arbitrary. Thus (\ref{L1 est eq7}) applies when $q=2$, $K=X\setminus D$ and $\delta=0$, that is, we obtain
$$
0\ge \int_{X\setminus D}\langle\Psi(u_{\infty})\partial^{\nabla_{h_0}}u_{\infty},\partial^{\nabla_{h_0}}u_{\infty}\rangle_{h_0,T}+\sqrt{-1}\int_{X\setminus D}\Tr(u_{\infty}\Lambda_TF_{h_0})T^n.
$$

(3) By $(1)$ and $(2)$, using Proposition \ref{weak proj} below, we can prove $(3)$ by the same argument with \cite[Section 5]{Sim88} (refer to the comment below \cite[Claim 4.6]{FO25}). We only sketch the proof. By \cite[Lemma 5.5]{Sim88}, we see that the eigenvalues $\lambda_1,\ldots,\lambda_r$ of $u_{\infty}$ are all constants and $\lambda_i$ are not all equal. Let $\{\gamma\}$ be the set of intervals between the eigenvalues of $u_{\infty}$ and for each $\gamma$, we choose $p_{\gamma}:\mathbb{R}\to \mathbb{R}$ with $p_{\gamma}(\lambda_i)=1$ for $\lambda_i<\gamma$ and $p_{\gamma}(\lambda_i)=0$ for $\lambda_i>\gamma$. Define $\pi_{\gamma}:=p_{\gamma}(u_{\infty})$. Then, by \cite[Lemma 5.6]{Sim88}, we can see that $\pi_{\gamma}$ defines a $T$-weak holomorphic subbundle (see Definition \ref{weak subbundle defi} below). By \cite[Lemma 5.7]{Sim88}, there exists $\gamma$ such that
$$
\mu_{\alpha}(E)\le \frac{1}{\Tr(\pi_{\gamma})}\int_{X\setminus W}\left(\Tr(\pi_{\gamma}\sqrt{-1}F_{h_0})-\sqrt{-1}\Tr(\delbar\pi_{\gamma}\wedge(\delbar\pi_{\gamma})^{*_{h_0}})\right)\wedge T^{n-1}
$$
where $W\subset X$ is an analytic subset in $X$. By Proposition \ref{weak proj} below, there exists a torsion free subsheaf $\mathcal{F}\subset E$ on $X$ such that  
$$
\frac{1}{\Tr(\pi_{\gamma})}\int_{X\setminus W}\left(\Tr(\pi_{\gamma}\sqrt{-1}F_{h_0})-\sqrt{-1}\Tr(\delbar\pi_{\gamma}\wedge(\delbar\pi_{\gamma})^{*_{h_0}})\right)\wedge T^{n-1}=\mu_{\alpha}(\mathcal{F}).
$$
It contradicts to the assumption that $E$ is $\alpha^{n-1}$-slope stable. Hence we obtain a uniform estimate $C>0$ such that the following holds:
$$
\int_X\left(-\log|s_D|^2+|s_i|\right)\omega_i^n\le C.
$$
\end{proof}
\begin{defi}\label{weak subbundle defi}
Let $X$ be a compact K\"{a}hler manifold and $T$ be a closed positive $(1,1)$-current on $X$ which is smooth K\"{a}hler on $X\setminus D$ where $D$ is a complex analytic subset in $X$. Let $(E,h_0)$ be a holomorphic hermitian vector bundle on $X$. Then, a $T$-weak holomorphic subbundle of $E$ is an $h_0$-hermitian endomorphism $\pi$ defined almost everywhere on $X\setminus D$ such that $\pi^2=\pi$, $(1-\pi)\circ\delbar^E\pi=0$ and
$$
\int_{X\setminus D}|\pi|_{h_0}^2T^n+\int_{X\setminus D}|\nabla_{h_0}\pi|^2_{h_0, T}T^n<\infty.
$$
By \cite{UY86}, the image of a $T$-weak holomorphic subbundle $\pi$ uniquely defines a coherent subsheaf $\mathcal{F}$ of $E|_{X\setminus D}$, We denote $\mathcal{F}$ by $\im(\pi)$.
\end{defi}
Following proposition is the Chern-Weil formula for $T$-weak holomorphic subbundles.
\begin{prop}\label{weak proj}
Let $X$ be a compact K\"{a}hler manifold and $\alpha$ be a nef and big class on $X$ with $D=E_{nK}(\alpha)$ is a snc divisor. Let $T$ be a closed positive $(1,1)$-current with minimal singularities in $\alpha$ which is smooth K\"{a}hler on $X\setminus D$.  Let $(E,h_0)$ be a holomorphic hermitian vector bundle on $X$. Let $\pi$ be a $T$-weak holomorphic subbundle of $E$. Then the following holds.
\begin{enumerate}
\item There exists a torsion free coherent subsheaf $\mathcal{F}\subset E$ on $X$ such that $\mathcal{F}|_{X\setminus D}=\im(\pi)$.
\item There exists an analytic subset $W\subset X$ such that
$$
\int_Xc_1(\mathcal{F})\wedge\alpha^{n-1}
=\int_{X\setminus W}\left(\Tr(\pi\sqrt{-1}F_{h_0})-\Tr(\delbar\pi\wedge(\delbar\pi)^{*_{h_0}})\right)\wedge \langle T^{n-1}\rangle.
$$
\end{enumerate}
\end{prop}
\begin{proof}
Let $\pi$ be a $T$-weak holomorphic subbundle of $(E,h_0)$ and $\im(\pi)$ be the torsion free coherent sheaf of $E|_{X\setminus D}$ defined by the image of $\pi$. Let us consider the following short exact sequence of torsion free sheaves on $X\setminus D$:
\begin{equation}\label{weak proj eq1}
0\to \mathcal{F}\xrightarrow{\iota}E|_{X\setminus D}\xrightarrow{p}\mathcal{G}\to 0.
\end{equation}
Then $p$ defines an $\mathcal{O}_X$-linear morphism $\wtil{p}:=i_*p:E\to i_*\mathcal{G}$ where $i:X\setminus D\hookrightarrow X$ is a natural inclusion. Then $\wtil{\mathcal{G}}:=\im(\wtil{p})/\Tor$ is a torsion free coherent sheaf on $X$. In fact, the dual of the surjection $\wtil{p}: E\to \im(\wtil{p})/\Tor$ gives an injection of $\mathcal{O}_X$-module $\wtil{p}^*:(\im(\wtil{p})/\Tor)^*\to E^*$. Then, by Oka\rq{s} coherence theorem, we obtain that $(\im(\wtil{p})/\Tor)^*$ is a coherent sheaf on $X$. Hence $\im(\wtil{p})/\Tor$ extends to a torsion free coherent subsheaf $\wtil{\mathcal{G}}$ of $(\im(\wtil{p})/\Tor)^{**}$ on $X$. Then we obtain the following short exact sequence of torsion free coherent sheaves on $X$:
\begin{equation}\label{weak proj eq2}
0\to \wtil{\mathcal{F}}\xrightarrow{\wtil{\iota}}E\xrightarrow{\wtil{p}}\wtil{\mathcal{G}}\to0
\end{equation}
where the restriction of (\ref{weak proj eq2}) to $X\setminus D$ coincides with (\ref{weak proj eq1}). We can find an analytic subset $W\subset X$ such that $\im(\wtil{\iota})|_{X\setminus W}\subset E|_{X\setminus W}$ is a holomorphic subbundle. Let $\wtil{\pi}$ be an $h_0$-orthogonal projection from $E|_{X\setminus W}$ to $\wtil{\mathcal{F}}|_{X\setminus W}$. Then we have
\begin{equation}\label{weak proj eq3}
\wtil{\pi}|_{X\setminus W}=\pi \hbox{ almost everywhere in $X\setminus D$}.
\end{equation}
Hence, it suffices to show the following:
\begin{equation}\label{weak proj eq4}
\int_Xc_1(\wtil{\mathcal{F}})\wedge\alpha^{n-1}=\int_{X\setminus W}c_1(\wtil{\mathcal{F}}, h_0|_{\wtil{\mathcal{F}}})\wedge\langle T^{n-1}\rangle,
\end{equation}
since 
\begin{equation}\label{weak proj eq4.5}
c_1(\wtil{\mathcal{F}}, h_0|_{\wtil{\mathcal{F}}})=\Tr(\wtil{\pi}\sqrt{-1}F_{h_0})-\Tr(\delbar\wtil{\pi}\wedge(\delbar\wtil{\pi})^{*_{h_0}}).
\end{equation}
 Since (\ref{weak proj eq4}) is proved by the standard way (e.g. \cite[Remark 8.5]{Kob87}), we only describe the outline of the proof. We denote by $q:=\rk\wtil{\mathcal{F}}$. Let us consider
$$
\wtil{j}:=\wedge^q\wtil{\iota}:\det\wtil{\mathcal{F}}\to \bigwedge^qE.
$$
Then $\{\wtil{j}=0\}=W$. Since $\wtil{j}$ is an injection of a holomorphic line subbundle, the pull back
$$
u:=\wtil{j}^*(\wedge^qh_0)
$$
defines a smooth hermitian metric on $\det\wtil{\mathcal{F}}|_{X\setminus W}$. Let $\tau$ be a local holomorphic frame of $\det\wtil{\mathcal{F}}$ and $s_1,\ldots,s_r$ be a local holomorphic frame of $E$. Then we can denote by
$$
\wtil{j}(\tau)=\Sigma_{I}\tau^Is_I \hbox{ where $s_I=s_{i_1}\wedge\ldots\wedge s_{i_p}$ with $i_1<\cdots<i_p$}.
$$
Let $\wtil{u}$ be a smooth hermitian metric on $\det\wtil{\mathcal{F}}$ which is globally defined over $X$. Then we define a nonnegative function
$$
f:=\frac{u(\tau,\tau)}{\wtil{u}(\tau,\tau)}=\Sigma_{I,J}u_{IJ}\tau^I\overline{\tau^J}
$$
where $u_{IJ}=\wedge^qh_0(s_I,s_J)/\wtil{u}(\tau,\tau)$. Then we can see that $f$ is independent of the choice of $\tau$, that is, $f$ is globally defined on $X$, and $\{f=0\}=W$. Let $\mu:Y\to X$ be a composition of blow-ups such that $\mu^{-1}(W)$ is a snc divisor in $Y$. Since $\mu^*f$ vanishes along $\mu^{-1}(W)$, we have 
$$
dd^c\log\mu^*f=[W^{\prime}]
$$
where $[W^{\prime}]$ is an integral current supported on $\mu^{-1}(W)$ with nonnegative coefficients. We obtain that
\begin{equation}\label{weak proj eq5}
\int_{X\setminus (D\cup W)}dd^c\log f\wedge T^{n-1}=\int_{Y\setminus \mu^{-1}(D)\cup\mu^{-1}(W)}dd^c\log\mu^*f\wedge(\mu^*T)^{n-1}=0.
\end{equation}
Since $c_1(\det\wtil{\mathcal{F}},\wtil{u})=c_1(\det\wtil{\mathcal{F}},u)+dd^c\log f$ on $X\setminus W$, we obtain 
\begin{align}\label{weak proj eq6}
\int_Xc_1(\det\wtil{\mathcal{F}})\wedge\alpha^{n-1}
&=\int_Xc_1(\det\wtil{\mathcal{F}},\wtil{u})\wedge\langle T^{n-1}\rangle\notag\\
&=\int_{X\setminus (D\cup W)}c_1(\det\wtil{\mathcal{F}},\wtil{u})\wedge T^{n-1}\notag\\
&=\int_{X\setminus (D\cup W)}c_1(\det\wtil{\mathcal{F}},u)\wedge T^{n-1}\notag\\
&=\int_{X\setminus W}c_1(\det\wtil{\mathcal{F}},u)\wedge \langle T^{n-1}\rangle.
\end{align}
In the third line above, we used (\ref{weak proj eq5}). Since $u=\wtil{j}^*(\wedge^qh_0)$, we have 
\begin{equation}\label{weak proj eq7}
c_1(\det\wtil{\mathcal{F}},u)=c_1(\wtil{\mathcal{F}},h) \hbox{ on $X\setminus W$}.
\end{equation}
Then by (\ref{weak proj eq3}), (\ref{weak proj eq4.5}), (\ref{weak proj eq6}) and (\ref{weak proj eq7}), we obtain
\begin{align*}
\int_Xc_1(\wtil{\mathcal{F}})\wedge\alpha^{n-1}
&=\int_{X\setminus W}c_1(\wtil{\mathcal{F}},h_0|_{\wtil{\mathcal{F}}})\wedge\langle T^{n-1}\rangle\\
&=\int_{X\setminus (D\cup W)}\left(\Tr(\pi\sqrt{-1}F_{h_0})-\Tr(\delbar\pi\wedge(\delbar\pi)^{*_{h_0}})\right)\wedge T^{n-1}.
\end{align*}
\end{proof}
\subsubsection{End of the proof of Theorem \ref{main thm2 intro} (1) \texorpdfstring{$\Rightarrow$}{Rightarrow} (2)}
As a consequence of the previous subsections (refer to Lemma \ref{uniform est defi lemm}, Lemma \ref{C0 to L2 limit}, Lemma \ref{C0 to L1} and Lemma \ref{L1 est} ), we obtain the followings:
\begin{theo}\label{main theo}
We use the notations in Notation \ref{notations} and Lemma \ref{Uhl limit}. 
Then there exists a smooth $h_0$-hermitian endomorphism ${\Psi_{\infty}}$ of $E$ which satisfies the following conditions:
\begin{enumerate}
\item For any $0<\alpha<1$, there exists a subsequence of $\{{\Psi_i}\}_i$, also denoted by $\{{\Psi_i}\}_i$, such that for any compact subset $A\subset X\setminus D$, the following holds:
$$
\lim_{i\to\infty}\|{\Psi_i}-{\Psi_{\infty}}\|_{C^{1,\alpha}(A)}=0.
$$
\item There exists a constant $C>0$ such that the following holds:
\begin{itemize}
\item $\sup_{X\setminus D}|s_D{\Psi_{\infty}}^{\pm1}|_{h_D\otimes h_0\otimes h_0^*}\le C.$
\item $\int_{X\setminus D}|s_D^{m+1}\partial^{\nabla_{h_0}}\Psi_{\infty}^{\pm1}|^2_TT^n\le C$.
\end{itemize}
\item Furthermore, ${\Psi_{\infty}}$ is nontrivial: ${\Psi_{\infty}}\not\equiv 0$.
\end{enumerate}
\end{theo}
\begin{proof}
Let us denote by $\Psi_i:=e^{s_i}$ for some $h_0$-hermitian endomorphism of $(E,h_0)$ with $\Tr(s_i)=0$. Then, by Lemma \ref{L1 est}, we obtain that 
$$
\int_X\left(-\log|s_D|^2+|s_i|_{h_0}\right)\le C.
$$
Then, by Lemma \ref{C0 to L1}, we have
$$
\sup_X|s_D\Psi_i|_{h_D\otimes h_0\otimes h_0^*}\le C.
$$
Hence, by Lemma \ref{C0 to L2 limit}, we obtain $\Psi_{\infty}$ as in the statement.
\end{proof}
Since $\Psi_i$ in Lemma \ref{Uhl limit} defines an $\omega_i$-HYM metric, we obtain the following by Theorem \ref{main theo}, which proves Theorem \ref{main thm2 intro} (1) $\Rightarrow$ (2):
\begin{theo}\label{KH corr nef big with esti} 
Let $X$ be a compact K\"{a}hler manifold and $\alpha$ be a nef and big class on $X$. Let $T$ be an adapted closed positive $(1,1)$-current in $\alpha$. Let $E$ be a holomorphic vector bundle on $X$. If $E$ is $\alpha^{n-1}$-slope stable, then $E$ admits a $T$-adapted HYM metric $h=h_0\Psi^2$ such that the following holds:
\begin{enumerate}
\item $\sup_{X\setminus D}|s_D\Psi^{\pm1}|\le C.$
\item $\int_{X\setminus D}|s_D^{m+1}\partial^{\nabla_{h_0}}\Psi^{\pm1}|^2T^n\le C.$
\item The $T$-HYM constant of $h$ satisfies
$$
\lambda=\frac{1}{\alpha^n}\frac{c_1(E)\cdot \alpha^{n-1}}{\rk E}.
$$
\end{enumerate}
\end{theo}
\subsection{Proof of Theorem \ref{main thm2 intro} (2) \texorpdfstring{$\Rightarrow$}{Rightarrow} (1) and the uniqueness}\label{HYM stable section}
\begin{notation}\label{notation3}
\begin{itemize}
Let $X$ be a compact K\"{a}hler manifold with a smooth K\"{a}hler metric $\omega_0$ on $X$ and $\alpha$ be a nef and big class on $X$. In this subsection, we use the following notations. 
\item  Assume that $D=E_{nK}(\alpha)$ is a snc divisor. We denote by $s_D$ a defining section of $D$ and $h_{D,0}$ a smooth hermitian metric on $D$ with $|s_D|_{h_{D,0}}\le 1$.
\item Let $T$ be a closed positive $(1,1)$-current in $\alpha$ satisfying $|s_D|^k\langle T^n\rangle\le C\omega_0^n$ for some $k>0$ (refer to Lemma \ref{adapted curr prop}). Let $m\in\mathbb{Z}_{\ge0}$ be a nonnegative integer so that $T\ge |s_D|^{2m}\omega_0$ holds.
\item Let $(E,h_0)$ be a pair consisting of a holomorphic vector bundle $E$ over $X$ and a smooth hermitian metric $h_0$ on $E$.
\end{itemize}
\end{notation}
In this section, the following lemma will repeatedly used. The proof is the same with Lemma \ref{Uhl limit}.
\begin{lemm}\label{adHYM hol}
We use the notations in Notation \ref{notation3}. Let $(F, h_F)$ be a holomorphic hermitian vector bundle on $X$. Assume that $F$ admits a $T$-adapted HYM metric $h=h_F\Psi^2$. Let us define a connection $\nabla$ of $F$ by $\nabla=\Psi\circ\nabla_{h}\circ \Psi^{-1}$. Then $\nabla$ is a $T$-admissible HYM connection of a complex hermitian vector bundle $(F,h_F)$. Furthermore, we have that the morphism $\Psi:(F,\delbar^F)\to (F,\delbar^{\nabla})$ is holomorphic.
\end{lemm}
Now, we recall the estimates for $T$-adapted HYM metrics of $\alpha^{n-1}$-slope stable vector bundles in Theorem \ref{KH corr nef big with esti} (1) and (2) which will be used in this section.
We first establish the following lemma.
\begin{lemm}\label{nK para}
We use notations in Notation \ref{notation3}. Fix $l>0$ sufficently large. Let $h_D:=h_{D,0}\Psi_D^2$ be a $T$-adapted HYM metric on $\mathcal{O}(lD)$ constructed in Theorem \ref{KH corr nef big with esti}. Let us consider the following composition of holomorphic morphisms:
$$
(\mathcal{O}_X,h_{\triv})\xrightarrow{s_D^l}(\mathcal{O}(lD), h_{D,0})\xrightarrow{\Psi_D}(\mathcal{O}(lD)|_{X\setminus D}, \nabla_D) 
$$
where $h_{\triv}$ is the trivial metric on $\mathcal{O}_X$ and $\nabla_D=\Psi_D\circ\nabla_{h_D}\circ\Psi_D^{-1}$ is the $T$-admissible HYM connection on $(\mathcal{O}(lD),h_{D,0})|_{X\setminus D}$. Then the following holds:
$$
(\nabla_{D}\otimes\nabla_{h_{\triv}}^*)(\Psi_Ds_D^l)=0 \hbox{ on $X\setminus D$}.
$$
\end{lemm}
\begin{proof}
We recall that, by the construction of $\Psi_D$ (see Theorem \ref{KH corr nef big with esti}) , it satisfies
\begin{equation}\label{nK para eq1}
\sup_{X\setminus D}|s_D\Psi_D|\le C \hbox{ and } \int_{X\setminus D}|s_D^{m+1}\partial^{\nabla_{h_{D,0}}}\Psi_D|^2\le C.
\end{equation} 
Since $s_D^l:(\mathcal{O}_X,\delbar^{\nabla_{h_{\triv}}})\to (\mathcal{O}(lD),\delbar^{\nabla_{h_{D,0}}})$ and $\Psi_D:(\mathcal{O}(lD),\delbar^{\nabla_{h_{D,0}}}) \to (\mathcal{O}(lD)|_{X\setminus D}, \delbar^{\nabla_D})$ is holomorphic (refer to Lemma \ref{adHYM hol}) and since $\sqrt{-1}\Lambda_TF_{h_D}=0$ by Theorem \ref{KH corr nef big with esti}, we obtain
\begin{equation}\label{nK para eq2}
\Delta_T|\Psi_Ds_D^l|_{h_{D,0}\otimes h_{\triv}}^2=\left|(\nabla_{D}\otimes\nabla_{h_{\triv}}^*)(\Psi_Ds_D^l)\right|^2_{h_{D,0}\otimes h_{\triv}^*,T}.
\end{equation}
We show that the integral of LHS with respect to $T^n$ equals to 0. Let $\eta_k:X\to [0,1]$ be a smooth cut-off functions as in Lemma \ref{cut-off}, that is, each $\eta_k$ satisfies
$$
\Supp(\eta_k)\Subset X\setminus D,\hspace{2mm} X\setminus D_{\varepsilon_{k-1}}\subset \{\eta_k=1\},\hspace{2mm} \lim_{k\to\infty}\int_X|d\eta_k|^2_{\omega_0}\omega_0^n=0.
$$
Then we comute as follows:
\begin{align*}
\left|\nabla|\Psi_Ds_D^l|^2\right|_T
&\le 2\left|\langle(\nabla_D\otimes\nabla_{h_{D,0}}^*\Psi_D)s_D^l+l\Psi_Ds_D^{l-1}\nabla_{h_{D,0}}s_D,\Psi_Ds_D^l\rangle\right|\\
&\le |s_D|^{m+1}|\partial^{\nabla_{h_{D,0}}}\Psi_D|_T|s_D\Psi_D||s_D|^{2l-m-2} +l|s_D\Psi_D|^2|s_D|^{2m}|\nabla_{h_{D,0}}s_D|_T|s_D|^{2l-2m-3}\\
&\le C\left(|s_D^{m+1}\partial^{\nabla_{h_{D,0}}}\Psi_D|_T+1 \right)|s_D|^{2l-2m-3}
\end{align*}
In the last line, we used (\ref{nK para eq1}) and the assumption that $T\ge |s_D|^{2m}\omega_0$. Then we can estimate as follows:
\begin{align*}
\left|\int_{X\setminus D}\eta_k\Delta_T|\Psi_Ds_D^l|^2T^n \right|
&=\left|\int_{X\setminus D}\langle\nabla\eta_k,\nabla|\Psi_Ds_D^l|^2\rangle_TT^n \right|\\
&\le \int_{X\setminus D}|\nabla\eta_k|_T|\nabla|\Psi_Ds_D^l|^2|_TT^n\\
&\le C\left(\int_X|\nabla\eta_k|_{\omega_0}^2|s_D|^{4l-2m-6}T^n\right)^{1/2}\left(\int_{X\setminus D}|s_D^{m+1}\partial^{\nabla_{h_{D,0}}}\Psi_D|_T^2T^n \right)^{1/2}\\
&\hspace{4mm}+C\int_X|\nabla\eta_k|_{\omega_0}|s_D|^{2l-4m-3}T^n\\
&\le  C\left(\int_X|\nabla\eta_k|_{\omega_0}^2\omega_0^n\right)^{1/2}\left(\int_{X\setminus D}|s_D^{m+1}\partial^{\nabla_{h_{D,0}}}\Psi_D|_T^2T^n \right)^{1/2}\\
&\hspace{4mm}+C\int_X|\nabla\eta_k|_{\omega_0}\omega_0^n\\
\end{align*}
In the last line, we use $|s_D|^{k}T^n\le C\omega_0^n$ (see Lemma \ref{adapted curr prop}). Hence we obtain
\begin{equation}\label{nK para eq3}
\int_{X\setminus D}\Delta_T|\Psi_Ds_D^l|^2T^n
=\lim_{k\to\infty}\int_{X\setminus D}\eta_k\Delta_T|\Psi_Ds_D^l|^2T^n=0.
\end{equation}
Then, by (\ref{nK para eq2}) and (\ref{nK para eq3}), we obtain that
$$
(\nabla_D\otimes\nabla_{h_{\triv}}^*)\Psi_Ds_D^l=0 \hbox{ on $X\setminus D$}.
$$
\end{proof}
\begin{theo}\label{unique conn}
If $(E,h_0)$ admits two $T$-adapted HYM metric $h_1$ and $h_2$, then the Chern connections $\nabla_{h_1}$ of $h_1$ coincides with the Chern connection $\nabla_{h_2}$ of $h_2$, that is, we have $\nabla_{h_1}=\nabla_{h_2}$. 
\end{theo}
\begin{proof}
Let $h_1=h_0\Psi_1^2$ and $h_2=h_0\Psi_2^2$ be two $T$-adapted HYM metrics with $\det\Psi_i=1$ for $i=1,2$. Suppose that the $T$-HYM constants $\lambda_i$ of $h_i$ satisfies $\lambda_1\ge \lambda_2$. Let us consider a connection $\nabla_i:=\Psi_i\circ\nabla_{h_i}\circ\Psi_i^{-1}$ on $E$. Then $\nabla_i$ is a $T$-admissible HYM connection of the complex hermitian vector bundle $(E,h_0)$ and 
$$
\Psi_i: (E,\delbar^E)|_{X\setminus D}\to (E|_{X\setminus D},\delbar^{\nabla_i})
$$  
is a holomorphic isomorphism (see Lemma \ref{adHYM hol}) Furthermore, by definition of $T$-adapted HYM metrics, we have that for $m$ in Notation \ref{notation3},
\begin{equation}\label{unique conn eq1}
\sup_{X\setminus D}|s_D^{l}\Psi_i|_{h_0}\le C \hbox{ for some $l\in\mathbb{Z}_{\ge m}$}
\end{equation}
and
\begin{equation}\label{unique conn eq2}
\int_{X\setminus D}\left|(\nabla_i\otimes\nabla_{h_0}^*)(s_D^l\Psi_i)\right|^2_{h_0^*\otimes h_0\otimes h_{D,0}^{l},T}T^n\le C.
\end{equation}
By Theorem \ref{KH corr nef big with esti}, the holomorphic line bundle $\mathcal{O}((4l+1)D)$ admits a $T$-adapted HYM metric $h_{D}=h_{D,0}\Psi_{D}^2$ with $\det \Psi_D=1$. Let us define $\nabla_D:=\Psi_D\circ\nabla_{h_D}\circ\Psi_D^{-1}$ a $T$-admissible HYM connection of $(\mathcal{O}((4l+1)D),h_{D,0})$. Then, by Theorem \ref{KH corr nef big with esti}, we have that $\Psi_D$ satisfies 
\begin{equation}\label{unique conn eq3}
\sup_{X\setminus D}|s_D\Psi_{D}|_{h_{D,0}^{4l+1,*}\otimes h_{D,0}^{4l+1}\otimes h_{D,0}}
\end{equation}
and 
\begin{equation}\label{unique conn eq4} 
\int_{X\setminus D}\left|s_D^{m+1}\partial^{\nabla_{h_{D,0}}}\Psi_{D}\right|^2_{h_{D,0}^*\otimes h_{D,0}\otimes h_{D,0}^{m+1},T}T^n\le C.
\end{equation}
Let us consider the following diagram. All vector bundles are restricted to $X\setminus D$:
\begin{center}
\begin{tikzpicture}[auto]
\node(e0) at (0,0) {$(E,\delbar^E,\nabla_{h_0},h_0)$};
\node(e1) at (8,0) {$(E,\delbar^{\nabla_1},\nabla_1,h_0)$};
\node(e2) at (0,-2) {$(E,\delbar^{\nabla_2},\nabla_2,h_0)$};
\node(ed0) at (4,-4) {$(E\otimes\mathcal{O}((4l+1)D),\delbar^{\nabla_2\otimes\nabla_{h_{D,0}}}, \nabla_2\otimes\nabla_{h_{D,0}} ,h_0\otimes h_{D,0})$};
\node(ed) at (8,-2) {$(E\otimes\mathcal{O}((4l+1)D),\delbar^{\nabla_2\otimes\nabla_D}, \nabla_2\otimes\nabla_{D} ,h_0\otimes h_{D,0})$};
\draw[->] (e1) to node[swap] {\scriptsize $\Psi_1^{-1}$} (e0);
\draw[->] (e0) to  node[swap] {\scriptsize $\Psi_2$} (e2);
\draw[->] (e2) to node[swap] {\scriptsize $\id_E\otimes s_D^{4l+1}$} (ed0) ;
\draw[->] (ed0) to node[swap] {\scriptsize $\id_E\otimes\Psi_D$} (ed);
\draw[->] (e1) to node {\scriptsize $s$} (ed);)
\end{tikzpicture}
\end{center}
Here $s:=(\id_E\otimes \Psi_D)\circ (\id_E\otimes s_D^{4l+1})\circ \Psi_2\circ\Psi_1^{-1}$ is a holomorphic morphism from $(E,\delbar^{\nabla_1})$ to $(E\otimes\mathcal{O}((4l+1)D),\delbar^{\nabla_2\otimes\nabla_D})$. We remark that the $T$-HYM constant of $\nabla_D$ equals to $0$ by Theorem \ref{KH corr nef big with esti}. Then the usual calculation shows the following point-wise equation on $X\setminus D$:
\begin{equation}\label{unique conn eq5}
\Delta_T|s|^2_{h_0\otimes h_0^*\otimes h_{D,0}}
=\left|(\nabla_2\otimes\nabla_1^*\otimes\nabla_D)s\right|^2_{h_0\otimes h_0^*\otimes h_{D,0},T}-(\lambda_2-\lambda_1)|s|^2_{h_0\otimes h_0^*\otimes h_{D,0}}.
\end{equation}
By (\ref{unique conn eq1}), (\ref{unique conn eq2}), (\ref{unique conn eq3}) and (\ref{unique conn eq4}), we can see that each term in RHS in (\ref{unique conn eq5}) is integrable over $X\setminus D$ with respect to $T^n$. Let $\eta_k$ be smooth cut-off functions on $X$ as in Lemma \ref{cut-off}, in particular they satisfies $\Supp(\eta_k)\Subset X\setminus D$, $\eta_k\nearrow 1$ and $\int_X|d\eta_k|^2_{\omega_0}\omega_0^n\to0$. Then, we obtain the following equation by the same way with (\ref{nK para eq3}):
\begin{equation}\label{unique conn eq6}
\int_{X\setminus D}\Delta_T|s|_{h_0\otimes h_0^*\otimes h_{D,0}}^2T^n=0.
\end{equation}
Here we recall that the assumption $\lambda_1\ge \lambda_2$. Then, by integrating both sides of (\ref{unique conn eq5}), we obtain
$$
\lambda_1=\lambda_2\hbox{ and } \left|(\nabla_2\otimes\nabla_1^*\otimes\nabla_D)s\right|^2_{h_0\otimes h_0^*\otimes h_{D,0},T}=0.
$$
In particular, we have
\begin{align}\label{unique conn eq7}
0
&=\left(\nabla_2\otimes\nabla_1^*\otimes\nabla_D\right)s\notag\\
&=\left(\nabla_2\otimes\nabla_1^*\otimes\nabla_D\right)(s_D^{4l+1}\circ\Psi_D\circ\Psi_2\circ\Psi_1^{-1})\notag\\
&=(\nabla_D\otimes\nabla_0^*)(s_D^{4l+1}\Psi_D)\circ\Psi_2\circ\Psi_1^{-1}
+s_D^{4l+1}\Psi_D\circ(\nabla_2\otimes\nabla_1^*)(\Psi_2\circ\Psi_1^{-1}).
\end{align}
By Lemma \ref{nK para}, we have $(\nabla_D\otimes\nabla_0^*)(s_D^{4l+1}\Psi_D)=0$. Then, by (\ref{unique conn eq7}), we have 
\begin{equation}\label{unique conn eq8}
(\nabla_2\otimes\nabla_1^*)(\Psi_2\circ\Psi_1^{-1})=0.
\end{equation} 
We remark that
\begin{equation}\label{unique conn eq8.5}
|\Psi_2\circ\Psi_1^{-1}|^2_{h_0\otimes h_0^*}=|\id_E|^2_{h_2\otimes h_1^*}
\end{equation}
and 
\begin{equation}\label{unique conn eq9}
\Delta_T|\id_E|^2_{h_2\otimes h_1^*}
=|(\nabla_{h_2}\otimes \nabla_{h_1}^*)(\id_E)|^2_{h_2\otimes h_1^*,T}-(\lambda_2-\lambda_1)|\id_E|^2_{h_2\otimes h_1^*},
\end{equation}
furthermore
\begin{equation}\label{unique conn eq10}
\Delta_T|\Psi_2\circ\Psi_1^{-1}|_{h_0\otimes h_0^*}
=|(\nabla_2\otimes\nabla_1^*)(\Psi_2\circ\Psi_1^{-1})|_{h_0\otimes h_0^*,T}
-(\lambda_2-\lambda_1)|\Psi_2\circ\Psi_1^{-1}|_{h_0\otimes h_0^*}
\end{equation}
at each point of $X\setminus D$.
Then by (\ref{unique conn eq8.5}), (\ref{unique conn eq9}) and (\ref{unique conn eq10}) together with (\ref{unique conn eq8}), we obtain 
\begin{equation}\label{unique conn eq11}
|(\nabla_{h_2}\otimes \nabla_{h_1}^*)\id_E|^2_{h_2\otimes h_1^*,T}
=\left|(\nabla_2\otimes\nabla_1^*)(\Psi_2\circ\Psi_1^{-1})\right|^2_{h_0\otimes h_0^*,T}=0.
\end{equation}
The equation (\ref{unique conn eq11}) shows $\nabla_{h_1}=\nabla_{h_2}$.
\end{proof}
\begin{corr}\label{HYM slope}
We use the notations in Notation \ref{notation3}. Assume that a holomorphic vector bundle $E$ admits a $T$-adapted HYM metric $h$. Then $h$ computes the $\alpha^{n-1}$-slope of $E$: 
$$
\mu_{\alpha}(E)=\frac{1}{\rk E}\int_{X\setminus D}c_1(E, h)\wedge \frac{T^{n-1}}{(n-1)!}
$$ 
where $c_1(E,h)=\sqrt{-1}Tr(F_h)$. In particular, the $T$-HYM constant $\lambda$ of $h$ is described as follows:
$$
\lambda=\frac{1}{\alpha^n}\frac{c_1(E)\cdot \alpha^{n-1}}{\rk E}.
$$
\end{corr}
\begin{proof}
Let $h$ be the $T$-adapted HYM metric on $E$. By Theorem \ref{KH corr nef big with esti}, the holomorphic line bundle $\det E$ admits a $T$-adapted HYM metric $h\rq{}$ whose $T$-HYM constant $\lambda\rq{}$ satisfies
$$
\lambda\rq{}=\frac{1}{\alpha^n}c_1(E)\cdot\alpha^{n-1}.
$$
Let $h$ be the given $T$-adapted HYM metric on $E$ with $T$-HYM constant $\lambda$. Then it induces a $T$-adapted HYM metric $\det h$ on $\det E$ with $T$-HYM constant $\rk E\lambda$. Then by Theorem \ref{unique conn}, we have $\nabla_{h\rq{}}=\nabla_{\det h}$, and thus the $T$-HYM constant $\lambda\rq{}$ of $\nabla_{h\rq{}}$ coincides with the $T$-HYM constant $\rk E\lambda$ of $\nabla_{\det h}$. Hence we obtain 
$$
\lambda=\frac{\lambda\rq{}}{\rk E}=\frac{1}{\alpha^n}\frac{c_1(E)\cdot\alpha^{n-1}}{\rk E}.
$$
Then we can prove the first equation of the statement by integrating the $T$-HYM equation $\sqrt{-1}\Lambda_TF_h=\lambda\id$ over $X\setminus D$ against $T^n$.
\end{proof}
We need the following lemma to prove the uniqueness of $T$-adapted HYM metrics.
\begin{lemm}\label{unique metric lem}
We use the notations in Notation \ref{notation3}. Assume that holomorphic vector bundles $E_1$ and $E_2$ are $\alpha^{n-1}$-slope stable with $\rk E_1=\rk E_2$ and $\mu_{\alpha}(E_1)=\mu_{\alpha}(E_2)$.
\begin{enumerate}
\item Let $g: E_1|_{X\setminus D}\to E_2|_{X\setminus D}$ be a nonzero holomorphic morphism satisfying 
$$
\sup_{X\setminus D}|s_D^lg|_{h_2\otimes h_1^*\otimes h_{D,0}}<\infty
$$
for some $l\in \mathbb{Z}_{\ge 0}$ where $h_i$ is a smooth hermitian metric on $E_i$ for $i=1,2$. Then $g$ is isomorphic.
\item In particular, if $E_1=E_2=E$, then there is a constant $a\in \mathbb{C}\setminus\{0\}$ such that $g=a\id_E$ on $X\setminus D$. That is, if $E$ is $\alpha^{n-1}$-slope stable, then 
$$
H^0\left(X\setminus D,\bigoplus_{k\in\mathbb{Z}}(\End(E)\otimes\mathcal{O}(kD))\right)=\mathbb{C},
$$
where $\mathcal{O}(kD)$ is regarded as a subsheaf of $\mathcal{O}_{X\setminus D}$ consisting of local rational functions with poles along $D$ at most of order $k$.
\end{enumerate}
\end{lemm}
\begin{proof}
Let us consider the image sheaf $\mathcal{G}$ of $s_D^lg$, which is a coherent subsheaf of $E_2\otimes\mathcal{O}(lD)|_{X\setminus D}$. Since $s_D^lg$ is bounded and $E_1$ is a locally free sheaf defined all over $X$, the coherent sheaf $\mathcal{G}$ on $X\setminus D$ is generated by bounded sections of $E_2\otimes\mathcal{O}(lD)|_{X\setminus D}$ around $D$. Then, since $E_2\otimes\mathcal{O}(lD)$ is defined all over $X$, 
 bounded generators of $\mathcal{G}$ extends across $D$ as a sections of $E_2\otimes\mathcal{O}(lD)|_{X\setminus D}$.
Hence, the image sheaf $\mathcal{G}$ extends to a coherent subsheaf of $E_2\otimes\mathcal{O}(lD)$, also denoted by $\mathcal{G}$. We remark that $s_D^lg\in H^0(X\setminus D,E_2\otimes E_1^*\otimes\mathcal{O}(lD))$ also extends to a holomorphic morphism defined all over $X$, since $s_D^lg$ is bounded and $E_1$ and $E_2\otimes\mathcal{O}(lD)$ are both defined over $X$. Then we obtain the following short exact sequence of coherent sheaves on $X$:
$$
0\rightarrow \mathcal{F}\rightarrow E_1\xrightarrow{s_D^lg} \mathcal{G}\rightarrow 0.
$$
Then, since $E_1$ and $E_2$ are $\alpha^{n-1}$-slope stable with same $\alpha^{n-1}$-slope, we can see that $\mathcal{F}=0$. In fact, the $\alpha^{n-1}$-slope stability of $E_2$ ensures that $E_2\otimes\mathcal{O}(lD)$ is also $\alpha^{n-1}$-slope stable. We recall that $\mathcal{G}\subset E_2\otimes\mathcal{O}(lD)$. Then if $\mathcal{F}\ne 0$, we have $\mu_{\alpha}(E_1)<\mu_{\alpha}(\mathcal{G})<\mu_{\alpha}(E_2\otimes\mathcal{O}(lD))=\mu_{\alpha}(E_1)$, that is a contradiction.
Therefore $s_D^lg:E_1\to E_2\otimes\mathcal{O}(lD)$ is injective.
Hence its determinant gives the short exact sequence
$$
0\to\det E_1\xrightarrow{\det(s_D^lg)}\det (E_2\otimes\mathcal{O}(lD))\to \tau_V\to 0.
$$
Since $s_D^lg$ is injective, the quotient sheaf $\tau_V$ is a torsion sheaf supported on some analytic subvariety $V\subset X$. We recall the assumptions that $\rk E_1=\rk E_2$ and $\mu_{\alpha}(E_1)=\mu_{\alpha}(E_2)$. Then we have
$$
0=c_1(\det (E_2\otimes\mathcal{O}(lD)))\alpha^{n-1}-c_1(\det E_1)\alpha^{n-1}
=c_1(\tau_V)\alpha^{n-1}.
$$
Hence we obtain a decomposition $V=V_1\cup V_2$ such that $V_1\subset D$ and $\codim V_2\ge 2$.  In particular we have that $\det(s_D^lg)$ is isomorphic over $X\setminus D$ and thus it has no nontrivial zero locus in $X\setminus D$. As a consequence $g: E_1|_{X\setminus D}\to E_2|_{X\setminus D}$ is isomorphic. Assume that $E=E_1=E_2$ and let $a\in \mathbb{C}\setminus\{0\}$ be an eigenvalue of $g$ at $x\in X\setminus D$. If $g-a\id_E\ne 0$, then it is an isomorphism by the above proof. But it is a contradiction. Hence $g=a\id_E$.
\end{proof}
Then we can prove the uniqueness of a $T$-adapted HYM metric on an $\alpha^{n-1}$-slope stable vector bundle.
\begin{lemm}\label{unique metric}
We use the notations in Notation \ref{notation3}.
If $E$ is $\alpha^{n-1}$-slope stable, then a $T$-adapted HYM metric on $E$ is unique up to scaling, if it exists.
\end{lemm}
\begin{proof}
Let $h_1$ and $h_2$ be two $T$-adapted HYM metrics on a holomorphic vector bundle $E$. We denote by $h_2=h_1g$ for some positive definite $h_1$-hermitian endomorphism $g$. Let $l\in\mathbb{Z}_{>0}$ be a sufficiently large integer.
Let $h_D$ be a $T$-adapted HYM metric on $\mathcal{O}(lD)$ whose existence is proved in Theorem \ref{KH corr nef big with esti}.
Let us define a smooth section $\overline{s_D^l}:=h_D(\cdot,s_D^l)$ of $\mathcal{O}(lD)^*=\mathcal{H}om(\mathcal{O},\mathcal{O}(lD)^*)$ over $X\setminus D$.  If we denote by $h_D=h_{D,0}\Psi_{D}^2$, then $\overline{s_D^l}=h_{D,0}(\Psi_D\circ\cdot,\Psi_Ds_D^l)$ and we have, by applying Lemma \ref{nK para} in the first equality below, 
\begin{align*}
0
&=(\nabla_D\otimes \nabla_{triv}^*)(\Psi_Ds_D^l)\\
&=(\Psi_D\circ\nabla_{h_D}\circ\Psi_D^{-1})\circ \Psi_Ds_D^l-\Psi_Ds_D^l\circ\nabla_{triv}\\
&=\Psi_D\circ(\nabla_{h_D}\otimes\nabla_{triv}^*) s_D^l\\
&=\Psi_D\circ(\nabla_{h_D}s_D^l)
\end{align*}
Here $\nabla_D=\Psi_D\circ\nabla_{h_D}\circ\Psi_D^{-1}$ the $T$-adapted HYM connection on $(\mathcal{O}(lD),h_{D,0})$ and $\nabla_{triv}$ is the trivial connection on the trivial bundle $\mathcal{O}$. Since $\Psi_D$ does not have zero locus on $X\setminus D$, we obtain
\begin{equation}\label{unique metric eq-1}
\nabla_{h_D}s_D^l=0.
\end{equation}
Then we also obtain that for any smooth section $t$ of $\mathcal{O}(lD)$,
\begin{align}\label{unique metric eq-0.5}
(\nabla_{h_D}\overline{s_D^l})(t)
&=d\left(h_D(t,s_D^l)\right)-h_D(\nabla_{h_D}t,s_D^l)\notag\\
&=-h_D(t,\nabla_{h_D}s_D^l)=0
\end{align}
by (\ref{unique metric eq-1}).
Then, we can see the following:
\begin{claim}\label{unique metric claim}
$\delbar^{E\otimes E^*\otimes\mathcal{O}(lD)^*}\left(\overline{s_D^l}g\right)=0$. In particular, we have
$$
\delbar^{E\otimes E^*}g=0.
$$
\end{claim}
\begin{proof}[proof of the claim]
In fact, by the uniqueness of a $T$-adapted HYM connection of $E$ (Theorem \ref{unique conn}), we have 
$$
0=\nabla_{h_1}-\nabla_{h_2}=g^{-1}\partial^{\nabla_{h_1}\otimes\nabla_{h_1}^*}g.
$$
In particular we have 
\begin{equation}\label{unique metric eq0}
\partial^{\nabla_{h_1}\otimes\nabla_{h_1}^*}g=0
\end{equation}
on $X\setminus D$ since $g$ is positive definite.  If we denote the $T$-adapted HYM metric $h_D$ on $\mathcal{O}(lD)$ by $h_D=h_{D,0}\Psi_{D}^2$, then we can assume $|s_D\Psi_D|_{h_{D,0}\otimes h_{D,0}^*}$ is bounded by Theorem \ref{KH corr nef big with esti}. If we further remark that $h_1$ and $h_2$ are both $T$-adapted, for sufficiently large $l\in \mathbb{Z}_{>0}$ we have 
\begin{equation}\label{unique metric eq1}
\sup_{X\setminus D}|\overline{s_D^l}g|_{h_0\otimes h_0^*\otimes h_{D,0}}<\infty,
\end{equation}
\begin{equation}\label{unique metric eq2}
\sup_{X\setminus D}|\overline{s_D^l}g|_{h_1\otimes h_1^*\otimes h_{D}}<\infty,
\end{equation}
\begin{equation}\label{unique metric eq3}
\int_{X\setminus D}\left|(\nabla_{h_1}\otimes\nabla_{h_1}^*\otimes\nabla_{h_D})\overline{s_D^l}g\right|^2_{h_1\otimes h_1^*\otimes h_D,T}T^n<\infty.
\end{equation}
If we use $\partial^{\nabla_{h_D}}\overline{s_D^l}=0$ (\ref{unique metric eq-0.5}) together with (\ref{unique metric eq0}), we have
\begin{equation}\label{unique metric eq3.5}
\partial^{\nabla_{h_1}\otimes\nabla_{h_1}^*\otimes\nabla_{h_D}}(\overline{s_D^l}g)=0.
\end{equation}
Thus, together with the facts that $h_1$ is $T$-HYM and the $T$-HYM constant of $h_D$ equals to 0 (refer to Theorem \ref{KH corr nef big with esti}), we obtain
\begin{equation}\label{unique metric eq4}
\Delta_T|\overline{s_D^l}g|^2_{h_1\otimes h_1^*\otimes h_D}=-2|\delbar^{E\otimes E^*\otimes\mathcal{O}(lD)^*}\overline{s_D^l}g|_{h_1\otimes h_1^*\otimes h_D,T}^2.
\end{equation}
by a standard computation. By (\ref{unique metric eq2}) and (\ref{unique metric eq3}) we can see that the both sides of  (\ref{unique metric eq4}) are integrable over $X\setminus D$ with respect to $T^n$ and by the same computation with (\ref{nK para eq3}) we obtain 
$$
0=\int_{X\setminus D}\Delta_T|\overline{s_D^l}g|^2_{h_1\otimes h_1^*\otimes h_D}T^n
=-2\int_{X\setminus D}|\delbar^{E\otimes E^*\otimes\mathcal{O}(lD)^*}\overline{s_D^l}g|_{h_1\otimes h_1^*\otimes h_D,T}^2T^n.
$$
Hence we have
\begin{equation}\label{unique metric eq5}
\delbar^{E\otimes E^*\otimes\mathcal{O}(lD)^*}\left(\overline{s_D^l}g\right)=0.
\end{equation}
Then, by (\ref{unique metric eq5}) and (\ref{unique metric eq-0.5}), we obtain that 
$$
\delbar^{E\otimes E^*}g=0.
$$
The proof of Claim \ref{unique metric claim} ends.
\end{proof}
Since $h_1$ and $h_2$ are $T$-adapted and $g$ is defined by $h_2=h_1g$, there is $l\in \mathbb{Z}_{\ge 0}$ such that $\sup_{X\setminus D}|s_D^lg|_{h_0\otimes h_0^*\otimes h_{D,0}}<\infty$. Then, by Claim \ref{unique metric claim} and Lemma \ref{unique metric lem}, we can conclude that $g=a\id_E$ over $X\setminus D$ for some $a\in\mathbb{R}\setminus \{0\}$ since now we assume that $E$ is $\alpha^{n-1}$-slope stable. As a consequence, we obtain $h_2=ah_1$ and thus we obtain the uniqueness of $T$-adapted HYM metrics on $E$.
\end{proof}
\begin{theo}\label{HYM stable}
We use the notations in Notation \ref{notation3}. If a holomorphic vector bundle $E$ admits a $T$-adapted HYM metric, then $E$ is $\alpha^{n-1}$-slope polystable.
\end{theo}
\begin{proof}
We denote by $h_E$ the $T$-adapted HYM metric on $E$.
Let $\iota:\mathcal{F}\hookrightarrow E$ be an $\alpha^{n-1}$-slope stable torsion free subsheaf of rank $q$ with torsion free quotient $\mathcal{G}:=E/\mathcal{F}$ such that 
\begin{equation}\label{HYM stable eq1}
\mu_{\alpha}(E)\le \mu_{\alpha}(\mathcal{F}). 
\end{equation}
Let $\pi:Y\to X$ be a sequence of blow-ups so that $F:=\pi^*\mathcal{F}/\Tor$ is locally free. Then we have $\mu_{\pi^*\alpha}(F)=\mu_{\alpha}(\mathcal{F})\ge\mu_{\alpha}(E)=\mu_{\pi^*\alpha}(\pi^*E)$. We denote by $i:F\to \pi^*E$ the induced morphism from the inclusion $\iota:\mathcal{F}\hookrightarrow E$. We remark that $\pi^*T$ is an adapted current on $Y$ (Lemma \ref{adapted app}), and $F$ is $\pi^*\alpha^{n-1}$-slope stable (Lemma \ref{inv stability}), $F$ admits a $T$-adapted HYM metric $h_F$ by Theorem \ref{KH corr nef big with esti}. Let $D^{\prime}:=E_{nK}(\pi^*\alpha)$ be a snc divisor in $Y$ given by the non-K\"{a}hler locus of $\pi^*\alpha$. We denote by $h_{D^{\prime}}$ the $\pi^*T$-adapted HYM metric on $\mathcal{O}(lD)$  for sufficiently large $l\in\mathbb{Z}_{>0}$. Then we consider the following composition of morphisms of $\pi^*T$-adapted HYM vector bundles:
$$
(F,h_F)\xrightarrow{i} (\pi^*E,\pi^*h_E)\xrightarrow{s_{D^{\prime}}^l} (\pi^*E\otimes\mathcal{O}(lD^{\prime}), \pi^*h_E\otimes h_{D^{\prime}}).
$$
Then, since $s_{D^{\prime}}^l\circ i$ is holomorphic, we have
\begin{align}\label{HYM stable eq2}
&\Delta_T|s_{D^{\prime}}^li|^2_{\pi^*h_E\otimes h_F\otimes h_{D^{\prime}}}\notag\\
&=|(\nabla_{\pi^*h_E}\otimes\nabla_{h_F}\otimes\nabla_{h_{D^{\prime}}})s_{D^{\prime}}^li|^2_{\pi^*h_E\otimes h_F\otimes h_{D^{\prime}},\pi^*T}
-(\lambda_E-\lambda_F)|s_{D^{\prime}}^li|^2_{\pi^*h_E\otimes h_F\otimes h_{D^{\prime}}}
\end{align}
where $\lambda_E$ is the $T$-HYM constant of $h_E$ and $\lambda_F$ is the $\pi^*T$-HYM constant of $h_F$. By Corollary \ref{HYM slope} and (\ref{HYM stable eq1}), we have $\lambda_E\le \lambda_F$. Since $\pi^*h_E$, $h_F$ and $h_{D^{\prime}}$ are all $\pi^*T$-adapted, each term in (\ref{HYM stable eq2}) is integrable over $Y\setminus D^{\prime}$ with respect to $T^n$ (refer to (\ref{unique conn eq5})). Furthermore, by the cut-off argument as (\ref{nK para eq3}), we obtain that 
$$
\int_{Y\setminus D^{\prime}}\Delta_T|s_{D^{\prime}}^li|^2_{\pi^*h_E\otimes h_F\otimes h_{D^{\prime}}}(\pi^*T)^n=0
$$
Then we obtain that $\lambda_E=\lambda_F$, in particular
\begin{equation}\label{HYM stable eq2.5}
\mu_{\alpha}(E)=\mu_{\alpha}(\mathcal{F}),
\end{equation}
and 
\begin{equation}\label{HYM stable eq3}
0
=(\nabla_{\pi^*h_E}\otimes\nabla_{h_F}\otimes\nabla_{h_{D^{\prime}}})s_{D^{\prime}}^li
=(\nabla_{h_{D^{\prime}}}s_{D^{\prime}}^l)i+ s_{D^{\prime}}^l(\nabla_{\pi^*h_E}\otimes\nabla_{h_F}^*)i.
\end{equation}
By (\ref{unique metric eq-1}), we have $\nabla_{h_{D^{\prime}}}s_{D^{\prime}}^l=0$. Hence (\ref{HYM stable eq3}) implies
\begin{equation}\label{HYM stable eq4}
(\nabla_{\pi^*h_E}\otimes\nabla_{h_F}^*)i=0.
\end{equation}
Let us consider a short exact sequence
\begin{equation}\label{HYM stable eq5}
0\to \mathcal{F}\xrightarrow{\iota}E\rightarrow\mathcal{G}\to 0
\end{equation}
over $X$ and a short exact sequence 
\begin{equation}\label{HYM stable eq6}
0\to F|_{Y\setminus D^{\prime}}\xrightarrow{i}\pi^*E|_{Y\setminus D^{\prime}}\rightarrow G\to 0
\end{equation}
over $Y\setminus D^{\prime}$. By (\ref{HYM stable eq4}), the sequence (\ref{HYM stable eq6}) holomorphically splits: $\pi^*E|_{Y\setminus D^{\prime}}\simeq F|_{Y\setminus D^{\prime}}\oplus G$. Since $F|_{Y\setminus D^{\prime}}=\pi^*(\mathcal{F}|_{X\setminus (D\cup\Sing\mathcal{F})})$ and $G=\pi^*(\mathcal{G}|_{X\setminus (D\cup\Sing\mathcal{F})})$, and the codimension of $\Sing\mathcal{F}$, the non locally free locus of $\mathcal{F}$, is at least 3, we obtain that the sequence (\ref{HYM stable eq5}) holomorphically splits: $E|_{X\setminus D}\simeq (\mathcal{F}\oplus\mathcal{G})|_{X\setminus D}$.
By (\ref{HYM stable eq2.5}), we have $\mu_{\alpha}(E)=\mu_{\alpha}(\mathcal{G})$. We remark that $\mathcal{G}$ also admits a $T$-adapted HYM metric, since $E$ and $\mathcal{F}$ both admit. Then, by repeating this argument by replacing $E$ to $\mathcal{G}$, we obtain the $\alpha^{n-1}$-slope polystability of $E$.
\end{proof}
Then we can prove the uniqueness in the general setting.
\begin{theo}\label{unique metric2}
We use the notations in Notation \ref{notation3}.
If a holomorphic vector bundle $E$ admits a $T$-adapted HYM metric, then it is unique up to scaling.
\end{theo}
\begin{proof}
Let $h_1$ and $h_2$ be two $T$-adapted HYM metrics on $E$. Then, by Theorem \ref{HYM stable}, we can see that $E$ is $\alpha^{n-1}$-slope polystable. That is, there exists $\alpha^{n-1}$-slope stable torsion free sheaves $\mathcal{E}_1,\ldots,\mathcal{E}_k$ on $X$ such that $E$ holomorphically splits to the direct sum of $\mathcal{E}_i$ on $X\setminus D$, that is, there is a holomorphic isomorphism $E|_{X\setminus D}\simeq (\mathcal{E}_1\oplus\cdots\mathcal{E}_k)|_{X\setminus D}$. By Theorem \ref{unique metric}, the $T$-adapted HYM metric $g_i$ on $\mathcal{E}_i$ is unique. As in the proof of Theorem \ref{HYM stable}, we can assume that $\iota:\mathcal{E}_1\hookrightarrow E$ is an $\alpha^{n-1}$-stable subsheaf and we have by (\ref{HYM stable eq4}), 
$$
(\nabla_{h_i}\otimes\nabla_{g_1}^*)\iota=0
$$ 
on the locally free locus of $\mathcal{E}_1$ for $i=1,2$. Thus $\nabla_{h_i}|_{\mathcal{E}_1}=\nabla_{g_1}$ and in particular $h_1|_{\mathcal{E}_1}=g_1=h_2|_{\mathcal{E}_1}$ by Theorem \ref{unique metric}. Since $\mathcal{E}_2=E/\mathcal{E}_1$, it admits three $T$-adapted HYM metrics $g_2$ and $\wtil{h_i}$ induced by $h_i$ and $h_i|_{\mathcal{E}_1}=g_1$ for $i=1,2$. Since $\mathcal{E}_2$ is $\alpha^{n-1}$-slope stable, we have that $\wtil{h_i}=g_2$ by Theorem \ref{unique metric}. Hence, inductively we obtain that 
$$
h_1=h_1|_{\mathcal{E}_1}\oplus\cdots\oplus h_1|_{\mathcal{E}_k}
=g_1\oplus \cdots\oplus g_k
=h_2|_{\mathcal{E}_1}\oplus\cdots\oplus h_2|_{\mathcal{E}_k}=h_2.
$$
\end{proof}
In summary, we obtain the following result that proves Theorem \ref{main thm2 intro} (2) $\Rightarrow$ (1) and the uniqueness.
\begin{theo}\label{HYM to stable and unique}
Let $X$ be a compact K\"{a}hler manifold, $\alpha$ be a nef and big class and $E$ be a holomorphic vector bundle on $X$. Let $T$ be an adapted closed positive $(1,1)$-current in $\alpha$. Assume that $E$ admits a $T$-adapted HYM metric. Then $E$ is $\alpha^{n-1}$-slope polystable. Furthermore, a $T$-adapted HYM metric is unique up to scaling.
\end{theo}
\begin{proof}
The $\alpha^{n-1}$-slope polystability is proved in Theorem \ref{HYM stable}. The uniqueness is showed in Theorem \ref{unique metric2}.
\end{proof}

\section{Applications}
\subsection{Asymptotic behavior of adapted HYM metrics toward $E_{nK}(\alpha)$}
The proof of the estimates of $T$-adapted HYM metrics in Theorem \ref{KH corr nef big with esti} and the uniqueness of $T$-adapted HYM metric Theorem \ref{unique metric}, we obtain the following estimates.
\begin{corr}\label{HYM est}
Let $X$ be a compact K\"{a}hler manifold and $\alpha$ be a nef and big class on $X$ such that $D=E_{nK}(\alpha)$ is a snc divisor. We denote by $s_D$ a defining section of $D$ and $h_D$ a smooth hermitian metric on $\mathcal{O}(D)$ with $|s_D|_{h_D}\le 1$. Let $T$ be an adapted closed positive $(1,1)$-current lying in $\alpha$. Assume that a holomorphic vector bundle $E$ admits a $T$-adapted HYM metric $h=h_0\Psi^2$ with $\det \Psi=1$. Then the following holds:
\begin{enumerate}
\item $\sup_{X\setminus D}|s_D\Psi|_{h_{D}\otimes h_0\otimes h_0^*}<\infty$.
\item We denote by $\Psi=e^S$ for some $h_0$-hermitian endomorphism $S$. Then $S$ is less singular that $\log|s_D|_{h_D}^2$ in the following sense: Let $\varepsilon_k>0$ be a sequence of constants so that $\varepsilon_k\searrow 0$. We denote by $D_{\varepsilon_k}$ the $\varepsilon_k$-neighborhood of $D$. Then
$$
\inf_k\sup_{D_{\varepsilon_k}}\frac{|S|_{h_0}}{-\log|s_D|^2}=0.
$$
\end{enumerate}
\end{corr}
\begin{proof}
Let $h=h_0\Psi^2$ be a $T$-adapted HYM metric on $E$. Then, by Theorem \ref{HYM stable}, the holomorphic vector bundle $E$ is $\alpha^{n-1}$-slope polystable. By Theorem \ref{KH corr nef big with esti}, we obtain a $T$-adapted HYM metric $h_{\infty}=h_0{\Psi_{\infty}}^2$ such that $\det{\Psi_{\infty}}=1$ and $\sup_{X\setminus D}|s_D{\Psi_{\infty}}|_{h_D\otimes h_0\otimes h_0^*}<\infty$. By the uniqueness of $T$-adapted HYM metrics (Theorem \ref{unique metric}), we have $\Psi=a{\Psi_{\infty}}$. Since $\det\Psi=1=\det{\Psi_{\infty}}$, the constant $a$ equals to $1$: $\Psi={\Psi_{\infty}}$. Hence we obtain $(1)$.

Fix any integer $k\in\mathbb{Z}_{>0}$. Then $h^{\otimes k}=h_0^{\otimes k}\Psi^{\otimes k}$ gives a $T$-adapted HYM metric on $E^{\otimes k}$ with $\det\Psi^{\otimes k}=1$. 
Hence, there exists a constant $C_k>0$ such that
\begin{equation}\label{HYM est eq1}
\sup_{X\setminus D}\log(|s_D|_{h_D}|\Psi|^k_{h_0\otimes h_0^*})
=\sup_{X\setminus D}\log|s_D\Psi^{\otimes k}|\le C_k.
\end{equation}
We denote by $\Psi=e^S$ for some $h_0$-hermitian endomorphism $S$. Let $s_1,\ldots, s_r$ be the eigenvalues of $S$ at $x\in X\setminus D$ and diagonalize $S$ at $x$ so that $S=\diag(s_1,\ldots,s_r)$. Then, by (\ref{HYM est eq1}) and by $|\Psi|^{2k}\ge e^{2ks_1}+\cdots+e^{2ks_r}$, we obtain 
\begin{align}\label{HYM est eq2}
C_k
&\ge \log|s_D|^2+\log(e^{2ks_1}+\cdots+e^{2ks_r})\notag\\
&\ge \log|s_D|^2+\log(e^{2ks_i})\notag\\
&\ge \log|s_D|^2+2ks_i
\end{align}
for any $i=1,\ldots,r$.
By Theorem \ref{KH corr nef big with esti}, the inverse morphism $\Phi=\Psi^{-1}$ also satisfies\\ $\sup_{X\setminus D}|s_D\Psi^{-1}|_{h_D\otimes h_0^*\otimes h_0}<\infty$. Then  by applying the computation of (\ref{HYM est eq2}) to $(\Psi^{\otimes k})^{-1}$, we obtain that for any $i=1,\ldots,r$,
\begin{equation}\label{HYM est eq3}
\log|s_D|^2-2ks_i\le C_k.
\end{equation}
By (\ref{HYM est eq2}) and (\ref{HYM est eq3}), we obtain that 
$k|s_i|\le C_k-\log|s_D|^2$. Hence we have
\begin{equation}\label{HYM est eq4}
k|S|_{h_0\otimes h_0^*}\le r(C_k-\log|s_D|^2).
\end{equation}
Let us fix a sequence of constants $\varepsilon_k>0$ such that $\varepsilon_k\searrow 0$ and such that 
$$
\frac{C_k}{-\log|s_D|^2}\le 1
$$
on the $\varepsilon_k$-neighborhood $D_{\varepsilon_k}$ of $D$. Then by (\ref{HYM est eq4}), we obtain
\begin{equation}\label{HYM est eq5}
\sup_k\sup_{D_{\varepsilon_k}}\frac{k|S|_{h_0\otimes h_0^*}}{-\log|s_D|^2}\le 2r.
\end{equation}
Here we suppose that $\inf_k\sup_{D_{\varepsilon_k}}\frac{|S|_{h_0\otimes h_0^*}}{-\log|s_D|^2}\ge c>0$ for some constant $c>0$. Then
$$
\infty=\sup_k(kc)\le\sup_k\sup_{D_{\varepsilon_k}}\frac{k|S|_{h_0\otimes h_0^*}}{-\log|s_D|^2},
$$
contradicting to (\ref{HYM est eq5}). Hence we obtain 
$$
\inf_k\sup_{D_{\varepsilon_k}}\frac{|S|_{h_0\otimes h_0^*}}{-\log|s_D|^2}=0.
$$
\end{proof}

\subsection{Proof of Theorem \ref{big BG eq intro}}
Projectively flatness is defined as follows.
\begin{defi}\cite[Definition 3.2 and Proposition 3.7]{GKP22}
\label{defn-MY-1}
Let $X$ be a connected complex manifold, and let $\mathcal{F}$ be a locally free coherent sheaf of rank $r$. We say that $\mathcal{F}$ is \emph{projectively flat} if the following equivalent conditions hold:
\begin{enumerate}
\item There exists a representation $\pi_1(X) \to PGL(r+1,\C)$ such that the associated projective bundle satisfies $\mathbb{P}(\mathcal{F}) \cong \mathbb{P}_{\rho}$, where
$$
 \mathbb{P}_{\rho} \cong X_{\mathrm{univ}} \times \mathbb{P}^r / \pi_{1}(X),
$$
and $X_{\mathrm{univ}}$ denotes the universal cover of $X$. 
\item The sheaf $\End(\mathcal{F})$ admits a flat connection. 
\item The sheaf $\Sym^r \mathcal{F} \otimes \det(\mathcal{F})^{\vee}$ admits a flat connection.
\end{enumerate}
\end{defi}
As usual, the notion of Jordan-H{$\rm \ddot{o}$}lder filtration for a nef and big class is defined as follows.
\begin{defi}\label{JH defi}
Let $X$ be a compact K\"{a}hler manifold and $\alpha$ be a nef and big class on $X$. Let $\mathcal{E}$ be an $\alpha^{n-1}$-slope semistable torsion free sheaf on $X$. Then, a Jordan-H${\rm \ddot{o}}$lder sequence of $\mathcal{E}$ is a filtration of $\mathcal{E}$ by saturated subsheaves
$$
0=\mathcal{E}_{k+1}\subset \mathcal{E}_k\subset \mathcal{E}_{k-1}\subset \cdots \subset \mathcal{E}_1\subset \mathcal{E}_0=\mathcal{E}
$$
such that each $\mathcal{E}_{i-1}/\mathcal{E}_i$ is $\alpha^{n-1}$-slope stable and $\mu_{\alpha}(\mathcal{E})=\mu_{\alpha}(\mathcal{E}_{i-1}/\mathcal{E}_i)$ for any $i$.
We define the graded sheaf associated to a Jordan-H${\rm \ddot{o}}$lder sequence by
$$
Gr^{JH}(\mathcal{E}):=\bigoplus_{i=1}^{k+1}(\mathcal{E}_{i-1}/\mathcal{E}_{i}).
$$
\end{defi}
By the standard argument and by Theorem \ref{main thm2 intro}, we can see the following proposition.
\begin{prop}\label{JH metric}
Let $X$ be a compact K\"{a}hler manifold and $\alpha$ be a nef and big class on $X$. Let $\mathcal{E}$ be a torsion free sheaf on $X$. Suppose that $\mathcal{E}$ is $\alpha^{n-1}$-slope semistable. Then the following holds.
\begin{enumerate}
\item $\mathcal{E}$ admits a Jordan-H{\rm${\ddot{o}}$}lder sequence.
\item The graded sheaf associated to a Jordan-H${\ddot{o}}$lder sequence of $\mathcal{E}$ admits a unique $T$-adapted HYM metric, where $T$ is an adapted closed positive $(1,1)$-current in $\alpha$.
\end{enumerate}
\end{prop}
\begin{proof}
By Lemma \ref{inv stability}, we can assume that $\mathcal{E}$ is locally free. \\
(1) The existence of a Jordan-H{$\rm\ddot{o}$}lder filtration is established in \cite[Corollary 4.13]{IJZ25}. We include the proof for the completeness. Assume that $\mathcal{E}$ is $\alpha^{n-1}$-slope semistable, but not stable,
Then, by Claim \ref{max slope} (2) (see also \cite[Lemma 4.4]{Jin25-2}),  there exists a saturated subsheaf $\mathcal{E}_1\subset \mathcal{E}$ such that
\begin{itemize}
\item[(a)] $\mu_{\alpha}(\mathcal{E}_1)=\mu_{\alpha}(\mathcal{E})$ and $0<\rk \mathcal{E}_1<\rk \mathcal{E}$.
\item[(b)] For any saturated subsheaf $\mathcal{F}\subset \mathcal{E}$ satisfying the two above conditions, then $\mathcal{E}_1\not\subset\mathcal{F}$ or $\mathcal{E}_1=\mathcal{F}$.
\end{itemize}
Then the rest of the proof of (1) is standard as follows. By (a) above, we have $\mu_{\alpha}(\mathcal{E}/\mathcal{E}_1)=\mu_{\alpha}(\mathcal{E})$. We prove that $\mathcal{E}/\mathcal{E}_1$ is $\alpha^{n-1}$-slope stable.
Let $\mathcal{F}\subset\mathcal{E}/\mathcal{E}_1$ be a saturated subsheaf with $0<\rk\mathcal{F}<\rk(\mathcal{E}/\mathcal{E}_1)$. We consider the following short exact sequences
\begin{center}
\begin{tikzpicture}[auto]
\node(011) at (0,0) {$0$}; \node(032) at (0,-1) {$0$};
\node(e11) at (2,0) {$\mathcal{E}_1$}; \node (e12) at (2, -1) {$\mathcal{E}_1$};
\node(e) at (4,0) {$\mathcal{E}$}; \node (ef) at (4,-1) {$\mathcal{E}_1+\mathcal{F}$};
\node(ef1) at (6, 0) {$\mathcal{E}/\mathcal{E}_1$}; \node(f) at (6,-1) {$\mathcal{F}$};
\node(012) at (8,0) {$0$}; \node(022) at (8,-1) {$0$.};
\draw[->] (011) to (e11); \draw[->] (e11) to (e); \draw[->] (e) to (ef1); \draw[->] (ef1) to (012);
\draw[->] (032) to (e12); \draw[->] (e12) to (ef); \draw[->] (ef) to (f); \draw[->] (f) to (022);
\draw[double, double distance=1pt] (e11) -- (e12);
\draw[->] (ef) to (e); \draw[->] (f) to (ef1);
\end{tikzpicture}
\end{center}
Since $\mathcal{E}$ is $\alpha^{n-1}$-slope semistable, we have $\mu_{\alpha}(\mathcal{E}_1+\mathcal{F})\le \mu_{\alpha}(\mathcal{E})$. Then by the exact sequence in the second line, we have
\begin{align*}
0
&=\rk \mathcal{F}(\mu_{\alpha}(\mathcal{E}_1+\mathcal{F})-\mu_{\alpha}(\mathcal{F}))+\rk \mathcal{E}_1(\mu_{\alpha}(\mathcal{E}_1+\mathcal{F})-\mu_{\alpha}(\mathcal{E}_1))\\
&\le \rk \mathcal{F}(\mu_{\alpha}(\mathcal{E})-\mu_{\alpha}(\mathcal{F}))+\rk \mathcal{E}_1(\mu_{\alpha}(\mathcal{E})-\mu_{\alpha}(\mathcal{E}_1))\\
&=\rk\mathcal{F}(\mu_{\alpha}(\mathcal{E}/\mathcal{E}_1)-\mu_{\alpha}(\mathcal{F})).
\end{align*}
Hence we have $\mu_{\alpha}(\mathcal{F})\le \mu_{\alpha}(\mathcal{E}/\mathcal{E}_1)$. If $\mu_{\alpha}(\mathcal{F})= \mu_{\alpha}(\mathcal{E}/\mathcal{E}_1)$, then the first line above shows 
$$
\rk\mathcal{F}(\mu_{\alpha}(\mathcal{E}_1+\mathcal{F})-\mu_{\alpha}(\mathcal{E}_1))=0.
$$
By (b) above, we obtain $\mathcal{E}_1+\mathcal{F}=\mathcal{E}$. Then, the above commutative diagram shows $\mathcal{F}=\mathcal{E}/\mathcal{E}_1$. It contradicts to $\rk\mathcal{F}<\rk(\mathcal{E}/\mathcal{E}_1)$. Therefore we have $\mu_{\alpha}(\mathcal{F})< \mu_{\alpha}(\mathcal{E}/\mathcal{E}_1)$ and thus $\mathcal{E}/\mathcal{E}_1$ is $\alpha^{n-1}$-slope stable. Then, inductively we obtain a Jordan-H${\rm \ddot{o}}$lder filtration of $\mathcal{E}$.\\
(2) By definition of Jordan-H${\rm \ddot{o}}$lder filtration, the graded sheaf associated to the Jordan-H${\rm \ddot{o}}$lder filtration is $\alpha^{n-1}$-slope polystable. Hence it admits a unique $T$-adapted HYM metric by Theorem \ref{main thm2 intro}.
\end{proof}
Then we prove the first main result of this subsection.
\begin{theo}\label{BG eq}
Let $X$ be a compact K\"{a}hler manifold and $\alpha$ be a nef and big class on $X$.
\begin{enumerate}
 \item Let $\mathcal{E}$ be an $\alpha^{n-1}$-slope polystable reflexive sheaf on $X$. Then $\mathcal{E}$ satisfies the following Bogomolov-Gieseker inequality:
\begin{equation}\label{BG ineq big}
\left(2rc_2(\mathcal{E})-rc_1(\mathcal{E})^2\right)\cdot\alpha^{n-2}\ge 0.
\end{equation}
Furthermore, if the equality of $(\ref{BG ineq big})$ holds, then $\mathcal{E}$ is locally free on the ample locus of $\alpha$, and $\mathcal{E}$ is projectively flat on the ample locus of $\alpha$.
\item Suppose that $\mathcal{E}$ is $\alpha^{n-1}$-slope semistable and it satisfies the Bogomolov-Gieseker equality
$$
\left(2rc_2(\mathcal{E})-rc_1(\mathcal{E})^2\right)\cdot\alpha^{n-2}= 0.
$$
Then a Jordan-H{\rm$\ddot{o}$}lder filtration of $\mathcal{E}$,
$$
0\subset \mathcal{E}_{k}\subset\cdots\subset\mathcal{E}_1\subset \mathcal{E},
$$
satisfies that $\mathcal{E}_{i}/\mathcal{E}_{i+1}|_{\Amp(\alpha)}$ is projectively flat.
\end{enumerate}
\end{theo}
\begin{proof}
(1) By the argument in \cite[case2 in the proof of Theorem 4.27]{IJZ25}, it suffices to show in the case that $\mathcal{E}$ is $\alpha^{n-1}$-slope stable (see also Remark \ref{nonopenness semistable}).
The proof is similar to \cite[\S 4.1.1]{Chen25}. Let $T$ be an adapted current in $\alpha$. Let $\pi:Y\to X$ be a composition of blow-ups along $E_{nK}(\alpha)$ and $\Sing(\mathcal{E})$ such that $E:=\pi^*\mathcal{E}/\Tor$ is locally free, $D=E_{nK}(\pi^*\alpha)$ is a snc divisor in $Y$ and $\pi^*T$ admits an adapted approximation $\omega_i\in\pi^*\alpha+(1/i)\{\omega_Y\}$ where $\omega_Y$ is a K\"{a}hler metric on $Y$.  By Lemma \ref{inv stability}, the vector bundle $E$ is $(\pi^*\alpha)^{n-1}$-slope stable. Then, by Lemma \ref{open stable}, we can find an $\omega_i$-HYM metric $h_i=h_0\Psi_i^2$ on $E$ with $\det \Psi_i=1$. And by Theorem \ref{KH corr nef big with esti}, we can also find a $T$-adapted HYM metric $h_{\infty}=h_0\Psi_{\infty}^2$ on $E$ with $\det \Psi_{\infty}=1$. Then, by Theorem \ref{main theo}, we have that $h_i$ subsequencially converges to $h_{\infty}$ in $C^{1,\alpha}$-topology on any relatively compact open subset in $X\setminus D$. 
Since
$$
\int_X\Tr(\sqrt{-1}F_{h_i}\wedge \sqrt{-1}F_{h_i})\wedge\omega_i^{n-2}
\le \int_X|F_{h_i}|_{h_i,\omega_i}\omega_i^n\le C,
$$
and 
$$
\int_X\left(\Tr(\sqrt{-1}F_{h_i})\wedge \Tr(\sqrt{-1}F_{h_i})\right)\wedge\omega_i^{n-2}
\le \int_X|F_{h_i}|_{h_i,\omega_i}\omega_i^n\le C,
$$
we can choose a further subsequence of $\{h_i\}_i$ so that 
\begin{equation}\label{BG eq eq1}
\Tr(\sqrt{-1}F_{h_i}\wedge \sqrt{-1}F_{h_i})\wedge\omega_i^{n-2}\to \Tr(\sqrt{-1}F_{h_{\infty}}\wedge \sqrt{-1}F_{h_{\infty}})\wedge T^{n-2}
\end{equation}
and
\begin{equation}\label{BG eq eq2}
\left(\Tr(\sqrt{-1}F_{h_i})\wedge \Tr(\sqrt{-1}F_{h_i})\right)\wedge\omega_i^{n-2}\to \left(\Tr(\sqrt{-1}F_{h_{\infty}})\wedge \Tr(\sqrt{-1}F_{h_{\infty}})\right)\wedge T^{n-2}
\end{equation}
on $Y\setminus D$ in the sense of current.
We denote by $c_i(h_i)$ the $i$-th Chern form defined by $h_i$. Since a sequence of measures
$$
d\mu_i:=(2rc_2(h_i)-(r-1)c_1(h_i)^2)\wedge\omega_i^{n-2}
$$
has uniformly bounded mass, we can find a subsequence such that $d\mu_i$ weakly converges to a measure $d\mu$ and, by (\ref{BG eq eq1}) and (\ref{BG eq eq2}), 
$$
d\mu=(2rc_2(h_{\infty})-(r-1)c_1(h_{\infty})^2)\wedge T^{n-2}
$$
on $Y\setminus D$. Since $d\mu$ has uniformly bounded mass on $Y$, the following computation works by the Fatou\rq{s} lemma:
\begin{align}\label{BG eq eq3}
(2rc_2(E)-(r-1)c_1(E)^2)\cdot(\pi^*\alpha)^{n-2}
&=\lim_{i\to\infty}(2rc_2(E)-(r-1)c_1(E)^2)\cdot\{\omega_i\}^{n-2}\notag\\
&=\lim_{i\to\infty}\int_{Y\setminus D}(2rc_2(h_i)-(r-1)c_1(h_i)^2)\wedge\omega_i^{n-2}\notag\\
&\ge\int_{Y\setminus D}(2rc_2(h_{\infty})-(r-1)c_1(h_{\infty})^2)\wedge T^{n-2}\notag\\
&\ge 0.
\end{align}
Since $\mathcal{E}$ is locally free in codimension 2, we have
$$
(2rc_2(\mathcal{E}-(r-1)c_1(\mathcal{E})^2)\cdot\alpha^{n-2}
=(2rc_2(E)-(r-1)c_1(E)^2)\cdot(\pi^*\alpha)^{n-2}.
$$ 
Hence $E=\pi^*\mathcal{E}/\Tor$ is projectively flat on $Y\setminus D$. Since $\pi:Y\setminus D\to X\setminus (E_{nK}(\alpha)\cup\Sing(\mathcal{E}))$ is isomorphic, we have that $\mathcal{E}|_{X\setminus (E_{nK}(\alpha)\cup\Sing(\mathcal{E}))}$ is projectively flat. That is, $\mathcal{E}$ is given by the representation $\rho:\pi_1(\Amp(\alpha)\setminus \Sing(\mathcal{E}))\to PU(n)$. Since $\codim\Sing(\mathcal{E})\ge 3$, we have $\pi_1(\Amp(\alpha)\setminus \Sing(\mathcal{E}))=\pi_1(\Amp(\alpha))$ and thus $\rho$ extends to the representation $\rho\rq{}:\pi_1(\Amp(\alpha))\to PU(n)$. Let $E\rq{}$ be the projectively flat vector bundle over $\Amp(\alpha)$ given by $\rho\rq{}$. Then $\mathcal{E}|_{\Amp(\alpha)\setminus \Sing(\mathcal{E})}\simeq E\rq{}|_{\Amp(\alpha)\setminus \Sing(\mathcal{E})}$. Since $\codim\Sing(\mathcal{E})\ge 3$ and $\mathcal{E}$ is reflexive, we have $\mathcal{E}|_{\Amp(\alpha)}\simeq E\rq{}$. In particular, we have that $\mathcal{E}$ is locally free on $\Amp(\alpha)$ and is projectively flat on $\Amp(\alpha)$.

(2) By Theorem \ref{JH metric}, we know that $\mathcal{E}$ admits a Jordan-H{\rm$\ddot{o}$}lder filtration 
$$
0\subset \mathcal{E}_{k}\subset\cdots\subset\mathcal{E}_1\subset \mathcal{E}.
$$
Then the argument in \cite[case 2 in the proof of Theorem 4.27]{IJZ25} shows that each $\mathcal{E}_{i}/\mathcal{E}_{i+1}$ satisfies the Bogomolov-Gieseker equality. By definition, each quotient $\mathcal{E}_{i}/\mathcal{E}_{i+1}$ is $\alpha^{n-1}$-slope stable. Hence, by (1) of this theorem, we see that $\mathcal{E}_{i}/\mathcal{E}_{i+1}$ is projectively flat on $X_{\reg}$.
\end{proof}
Next we discuss the case that $\alpha$ is a big class.
\begin{defi}[\cite{BH14}]\label{bimero zar decomp defi}
Let $X$ be a compact normal variety in Fujiki class and $\alpha\in H^{1,1}_{BC}(X)$ be a big class on $X$. Then, we say that $\alpha$ admits a bimeromorphic Zariski decomposition if there exists a resolution $\pi:Y\to X$ such that $\langle(\pi^*\alpha)\rangle$ is nef and big.
\end{defi}
Following \cite{Chen25}, we define the discriminant in this setting as follows:
\begin{defi}[cf. \cite{Chen25} section 1.3]\label{BG discri defi}
Let $X$ be a compact normal Moishezon variety and $\alpha\in H^{1,1}_{BC}(X)$ be a big class admitting a bimeromorphic Zariski decomposition. Let $\mathcal{E}$ be a reflexive sheaf on $X$. We define the discriminant as
\begin{equation}\label{BG discri}
\Delta(\mathcal{E})\cdot\langle\alpha^{n-2}\rangle
:=\inf_{\pi:Y\to X}\left(2rc_2(\pi^*\mathcal{E}/\Tor)-(r-1)c_1(\pi^*\mathcal{E}/\Tor)^2\right)\cdot\langle(\pi^*\alpha)^{n-2}\rangle,
\end{equation}
where $\pi:Y\to X$ is a resolution such that $Y$ is a smooth projective manifold, $\pi^*\mathcal{E}/\Tor$ is locally free and $\langle(\pi^*\alpha)\rangle$ is nef and big.
\end{defi}
The following proposition is the reason why we need to assume that varieties in Theorem \ref{big BG eq intro} are Moishezon.
\begin{lemm}[Theorem D in \cite{Nystr19}]\label{diff vol}
Let $X$ be a smooth projective manifold and $\alpha$ be a big class on $X$. Then a prime divisor $D$ is contained in the non-K\"{a}hler locus of $\alpha$ if and only if
$$
\langle\alpha^{n-1}\rangle\cdot[D]=0.
$$
\end{lemm}
Then we prove Theorem \ref{big BG eq intro}, a slight generalization of the previous Theorem \ref{BG eq}.
\begin{theo}[cf.\cite{IJZ25}]
\label{prop-semistable}
Let $X$ be a compact normal Moishezon variety, $\mathcal{E}$ be a reflexive sheaf of rank $r$ and $\alpha \in H^{1,1}_{BC}(X)$ be a big class admitting a bimeromorphic Zariski decomposition. 
\begin{enumerate}
\item If $\mathcal{E}$ is $\langle\alpha^{n-1}\rangle$-slope polystable, then the Bogomolov-Gieseker inequality 
\begin{equation}\label{prop-semistable eq}
\Delta(\mathcal{E})\langle\alpha^{n-2}\rangle\ge0
\end{equation}
holds. If $\Delta(\mathcal{E})\langle\alpha^{n-2}\rangle=0$ holds, then $\mathcal{E}$ is locally free on $\Amp(\alpha)$ and projectively flat on $\Amp(\alpha)$.
\item If $\mathcal{E}$ is $\langle\alpha^{n-1}\rangle$-slope semistable and it satisfies $\Delta(\mathcal{E})\langle\alpha^{n-2}\rangle=0$, then the Jordan-H{$\ddot{o}$}lder filtration of $\mathcal{E}$,
$$
0\subset \mathcal{E}_{k}\subset\cdots\subset\mathcal{E}_1\subset \mathcal{E},
$$
satisfies that $\mathcal{E}_{i}/\mathcal{E}_{i+1}|_{\Amp(\alpha)}$ is projectively flat.
\end{enumerate}
\end{theo}
\begin{proof}
As Theorem \ref{BG eq}, it suffices to show when $\mathcal{E}$ is $\langle\alpha^{n-1}\rangle$-slope stable.
Let $\pi:Y\to X$ be a resolution such that $Y$ is a smooth projective manifold and $\beta:=\langle\pi^*\alpha\rangle$ is a nef and big class.
Then, we have
$$
\langle(\pi^*\alpha)^k\rangle=\beta^k
$$
for any $k$ by \cite[Proposition 3.15]{IJZ25}.
By \cite[Proposition 2.5]{Tos18}, we have
$$
E_{nK}(\pi^*\alpha)=\pi^{-1}(E_{nK}(\alpha))\cup \Exc(\pi).
$$
Hence, using Lemma \ref{diff vol}, we have
$$
\beta^{n-1}\cdot[D\rq{}]=\langle(\pi^*\alpha)^{n-1}\rangle\cdot[D\rq{}]=0
$$
for any exceptional divisor $D\rq{}$.
 Thus, we can see that $\pi^{[*]}\mathcal{E}$ is $\beta^{n-1}$-slope stable as Lemma \ref{inv stability}. Let $h_{\infty}$ be a $T$-adapted HYM metric on $\pi^{[*]}\mathcal{E}$ where $T\in\beta$ is an adapted current. Then, by the definition of the discriminant (\ref{BG discri}) and the arugument in Theorem \ref{BG eq} (refer to \ref{BG eq eq3}), we have
$$
0=\Delta(\mathcal{E})\langle\alpha^{n-2}\rangle
\ge\int_{\Amp(\pi^*\alpha)}(2rc_2(h_{\infty})-(r-1)c_1(h_{\infty})^2)\wedge T^{n-2}\ge 0.
$$
Hence, we have that $\pi^{[*]}\mathcal{E}$ is projectively flat on $\Amp(\pi^*\alpha)$. Since $\pi:\Amp(\pi^*\alpha)\to\Amp(\alpha)$ is isomorphic and $E|_{\Amp(\pi^*\alpha)}=\pi^*(\mathcal{E}|_{\Amp(\alpha)})$, we obtain that $\mathcal{E}$ is projectively flat on $\Amp(\alpha)$. 

(2) Since now $X$ is Moishezon, $\mathcal{E}$ is $\langle\alpha^{n-1}\rangle$-slope semistable if and only if $\pi^*\mathcal{E}/\Tor$ is $\langle(\pi^*\alpha)^{n-1}\rangle$-slope semistable for any resolution $\pi:Y\to X$ so that $Y$ is smooth projective, $\pi^*\mathcal{E}/\Tor$ is locally free and $\langle\pi^*\alpha\rangle$ is nef and big by Lemma \ref{diff vol}. Then the push-forward of the Jordan-H{$\rm \ddot{o}$}lder filtration of $\pi^*\mathcal{E}/\Tor$ gives the Jordan-H{$\rm \ddot{o}$}lder filtration of $\mathcal{E}$,
$$
0\subset \mathcal{E}_{k}\subset\cdots\subset\mathcal{E}_1\subset \mathcal{E}.
$$
We recall that $\mathcal{E}_i/\mathcal{E}_{i+1}$ is $\langle\alpha^{n-1}\rangle$-slope stable. Then, by the argument in \cite[case 2 in the proof of Theorem 4.27]{IJZ25} and (\ref{prop-semistable eq}), we have
$$
0=\Delta(\mathcal{E})\langle\alpha^{n-2}\rangle\ge \Delta(\mathcal{E}_i/\mathcal{E}_{i+1})\langle\alpha^{n-2}\rangle\ge0.
$$
Then, by (1) in this theorem, we see that $\mathcal{E}_i/\mathcal{E}_{i+1}$ is projectively flat on $\Amp(\alpha)$.
\end{proof}
If $X$ is smooth and $\alpha$ is K\"{a}hler, then any $\alpha^{n-1}$-slope polystable reflexive sheaf $\mathcal{E}$ satisfying the equality of the Bogomolov-Gieseker inequality with respect to $\alpha$ is locally free on $X$. If $\alpha$ is nef and big, then one can in general, only asserts that $\mathcal{E}$ is locally free on the ample locus of $\alpha$. However, if the Chern classes of $\mathcal{E}$ vanish, then it can be shown that $\mathcal{E}$ is locally free globally on $X$:
\begin{corr}
Let $X$ be a compact K\"{a}hler manifold, $\alpha$ be a nef and big class and $\mathcal{E}$ be a reflexive sheaf on $X$. Assume that $\mathcal{E}$ is $\alpha^{n-1}$-slope polystable and $c_1(\mathcal{E})=0$, $c_2(\mathcal{E})=0$. Then $\mathcal{E}$ is locally free on $X$ and it is flat on $X$.
\end{corr}
\begin{proof}
As in the proof of Theorem \ref{BG eq} (1), we can assume that $\mathcal{E}$ is $\alpha^{n-1}$-slope stable.
Then, by Lemma \ref{open stable}, $\mathcal{E}$ is $(\alpha+\varepsilon\omega_0)^{n-1}$-slope stable. Since $c_1(\mathcal{E})=0$ and $c_2(\mathcal{E})=0$, we have $\Delta(\mathcal{E})\cdot(\alpha+\varepsilon\omega_0)^{n-2}=0$. Thus $\mathcal{E}$ is locally free on $X$ and flat globally on $X$.
\end{proof}
The following example demonstrates the existence of a vector bundle that does not satisfy the Bogomolov–Gieseker equality with respect to Kähler classes, but does satisfy it with respect to a nef and big class.
\begin{exam}\label{BG exam}
Let us consider $X$, $\alpha$, $\omega_{\varepsilon}$ and $E$ in Example \ref{nonopenness semistable}. Then $E$ is not $\omega_{\varepsilon}^2$-slope semistable, $\Delta(E)\cdot\omega_{\varepsilon}\neq0$ and $E$ is not projectively flat on $X$. On the other hand, in the limit $\varepsilon\to 0$, we obtain that $E$ is $\alpha^2$-slope polystable, $\Delta(E)\cdot\alpha=0$ and $E|_{X\setminus D}$ is projectively flat on $X\setminus D$.
\end{exam}

\subsection{Proof of Theorem \ref{MY eq intro}}
 In \cite{IJZ25}, the authors proved the following:
\begin{theo}[\cite{IJZ25} Theorem 1.1]
Let $X$ be an $n$-dimensional projective klt variety with big anti-canonical divisor $-K_X$ that is $K$-semistable. Then  the following Miyaoka-Yau inequality holds:
$$
\big(2(n+1)\widehat{c_2}(X)-n\widehat{c_1}(X)^2\big)\cdot\langle c_1(-K_X)^{n-2}\rangle\ge 0.
$$
Here $\widehat{c_i}(X)$ is the orbifold Chern classes of $X$. If $X$ is smooth in codimension 2, they coincide with usual Chern classes $c_i(X)$.
\end{theo}
In the following Theorem \ref{BG eq}, we show that the structure theorem holds when the equality of the Miyaoka-Yau inequality holds.

\begin{theo}[{cf. \cite{His24}}]\label{MY eq}
Let $X$ be an $n$-dimensional projective klt variety smooth in codimension $2$. Assume that $X$ is K-stable and that $-K_X$ is big. If the equality
\begin{equation}
\label{eq-MY-equality}
\bigl(2(n+1)c_2(X) - nc_1(X)^2 \bigr) \cdot \langle c_1(-K_{X})^{n-2} \rangle = 0
\end{equation}
holds, then the anticanonical model $Z$ admits a quasi-étale cover from $\mathbb{C}\mathbb{P}^n$. $($In this case, by \cite{Xu23}, the anticanonical model $Z$ exists.$)$
\end{theo}
We recall the canonical extension sheaf. 
Let $X$ be a normal analytic variety, and assume that $K_X$ is $\Q$-Cartier. Then, by \cite[Section~4]{GKP22} (see also \cite[Subsection~5.2]{IJZ25}), there exists a reflexive sheaf $\mathcal{E}_{X}$, called the \emph{canonical extension sheaf}, such that the following short exact sequence 
\begin{equation}
\label{eq-canoincal-extension}
0 \to \mathcal{O}_{X} \to \mathcal{E}_{X} \to \mathcal{T}_{X} \to 0
\end{equation}
is locally split, where $\mathcal{T}_{X}$ is the holomorphic tangent sheaf, i.e., the dual of the reflexive cotangent sheaf $\Omega_{X}^{[1]}$.

\begin{proof}
This proof is motivated by \cite{GKP22}. We use the same notation as in \cite[Subsection~5.2]{IJZ25}. Since $X$ is K-stable, it follows from \cite[Theorems 1.1, 1.2, and Corollary 3.5]{Xu23} that the anticanonical model $Z$ is a K-stable klt Fano variety. Moreover, by \cite[Remark 4 in \S 2 and Theorem 9 in \S 3]{DGP24}, the canonical extension sheaf $\mathcal{E}_Z$ is $c_1(-K_Z)^{n-1}$-polystable.

Let $f \colon X \dashrightarrow Z$ be the natural birational map, and take a resolution $W$ with morphisms $p \colon W \to X$ and $q \colon W \to Z$ resolving $f$:
$$
\xymatrix@C=40pt@R=30pt{
& W \ar[ld]_p \ar[rd]^q & \\
X \ar@{-->}[rr]^f && Z
}
$$
By the argument of \cite[Subsection~5.2]{IJZ25}, there exists an effective $q$-exceptional divisor $B$ on $W$ such that
\begin{equation}
\label{eq-contraction}
p^*(-K_X) = q^*(-K_Z) + B
\end{equation}
and $p(\Supp(B))$ coincides with the $f$-exceptional locus.
Therefore, we have an isomorphism 
\begin{equation}\label{eq-can sheaf}
p^*\mathcal{E}_X \cong q^*\mathcal{E}_Z \hbox{ outside the support of $B$}
\end{equation}
and $-K_X$ admits a birational Zariski decomposition. Since $c_i(X)=c_i(\mathcal{E}_X)$ and now $X$ is smooth in codimension 2, the assumption (\ref{eq-MY-equality}) coincides with
\begin{equation}\label{eq-MY is BG}
\Delta(\mathcal{E}_X)\langle(-K_X)^{n-2}\rangle=0.
\end{equation}
Thus, by Theorem \ref{prop-semistable}, we obtain that $\mathcal{E}_X$ is projectively flat on $X_{\reg}$. By (\ref{eq-can sheaf}) and the fact that $B$ is $q$-exceptional, we obtain
the following claim.
\begin{claim}
\label{claim-locallyfree}
The canonical extension sheaf $\mathcal{E}_Z$ is projectively flat on $Z_{\reg}$.
\end{claim}

Since $Z$ has at most klt singularities, by \cite[Theorem 1.14]{GKP16} (see also \cite[Subsection 2.5]{GKP22}), there exists a finite quasi-étale Galois cover $\nu \colon \widehat{Z} \to Z$ such that $\widehat{Z}$ is maximal quasi-étale. In particular, $-K_{\widehat{Z}} = \nu^{*}(-K_{Z})$ (refer to \cite[\S 5, (5.11)]{GKP22}), so $-K_{\widehat{Z}}$ is ample. 

We will show that $\widehat{Z} \cong \C\mathbb{P}^{n}$.
Let $\widehat{Z}^{\circ} := \nu^{-1}(Z_{\reg})$. Then $\mathcal{E}_{\widehat{Z}}|_{\widehat{Z}^{\circ}}$ is also projectively flat since $\nu \colon \widehat{Z}^{\circ} \to Z_{\reg}$ is étale by purity of the branch locus (refer to \cite[Remark 1.4]{GKP16}), and we have $\mathcal{E}_{\widehat{Z}}|_{\widehat{Z}^{\circ}} \cong \pi^{*}(\mathcal{E}_{Z}|_{Z_{\reg}})$ by \cite[Remark~4.3]{GKP22}. Hence by Definition~\ref{defn-MY-1}, the reflexive sheaf $\End(\mathcal{E}_{\widehat{Z}})$ is flat and locally free on $\widehat{Z}^{\circ}$. As $\codim(\widehat{Z} \setminus \widehat{Z}^{\circ}) \ge 2$, we have $\pi_1(\widehat{Z}_{\reg}) \cong \pi_1(\widehat{Z}^{\circ})$, so $\End(\mathcal{E}_{\widehat{Z}})$ is also flat locally free on $\widehat{Z}_{\reg}$. Since $\widehat{Z}$ is maximally quasi-étale, \cite[Theorem~1.14]{GKP16} implies that $\End(\mathcal{E}_{\widehat{Z}})$ is flat locally free on $\widehat{Z}$.

Consider the following locally split exact sequence:
\begin{equation}\label{can ext of Z}
0 \to \mathcal{O}_{\widehat{Z}} \to \mathcal{E}_{\widehat{Z}} \to \mathcal{T}_{\widehat{Z}} \to 0.
\end{equation}
Since (\ref{can ext of Z}) is locally splittable by the construction \cite[Construction 4.1]{GKP22}, we have that $\mathcal{T}_{\widehat{Z}}$ is a direct summand of $\End(\mathcal{E}_{\widehat{Z}})$. Since $\End(\mathcal{E}_{\widehat{Z}})$ is locally free, so is $\mathcal{T}_{\widehat{Z}}$. Thus, by \cite[Theorem 6.1]{GKKP11}, $\widehat{Z}$ is smooth.
Since $\mathcal{E}_{\widehat{Z}}$ is projectively flat and $\widehat{Z}$ is Fano, there exists a line bundle $L$ with 
$\mathcal{E}_{\widehat{Z}} \cong L^{\oplus (n+1)}$.
(Note that Fano manifolds are simply connected and thus the representation $\pi_1(\widehat{Z})\to PGL(n+1,\mathbb{C})$ corresponding to $\mathcal{E}_{\widehat{Z}}$ is the trivial representations.) Moreover, since $\det \mathcal{E}_{\widehat{Z}} \cong \mathcal{O}_{\widehat{Z}}(-K_{\widehat{Z}})$, we obtain
$$
\mathcal{O}_{\widehat{Z}}(-K_{\widehat{Z}}) \cong L^{\otimes (n+1)}.
$$
Since $-K_{\widehat{Z}}$ is ample, $L$ is also ample.
Therefore, by \cite[Corollary of Theorem 1.1]{KO73}, we conclude that $\widehat{Z} \cong \C\mathbb{P}^{n}$.
\end{proof}


\end{document}